\documentclass[11pt,a4paper,reqno,english]{amsart}
\usepackage[T1]{fontenc}
\usepackage{amsmath}
\usepackage{amsthm}
\usepackage{amsfonts}
\usepackage{amssymb}
\usepackage{graphicx}
\usepackage{amsbsy}
\usepackage{eucal}
\usepackage{amsbsy}
\usepackage{mathrsfs}
\usepackage{bbm}
\newtheorem{theorem}{Theorem}[section]
\newtheorem{lemma}[theorem]{Lemma}
\newtheorem{proposition}[theorem]{Proposition}

\newtheorem{definition}[theorem]{Definition}
\newtheorem{assumption}[theorem]{Assumption}
\theoremstyle{remark}

\numberwithin{equation}{section}
\catcode`@=11
\def\section{\@startsection{section}{1}%
  \z@{1.5\linespacing\@plus\linespacing}{.5\linespacing}%
  {\normalfont\bfseries\large\centering}}
\catcode`@=12
\newcommand{\be}{\begin{equation}}
\newcommand{\ee}{\end{equation}}
\newcommand{\bea}{\begin{eqnarray}}
\newcommand{\eea}{\end{eqnarray}}
\newcommand{\bee}{\begin{eqnarray*}}
\newcommand{\eee}{\end{eqnarray*}}

\def\pa{\partial}
\def\na{\nabla}

\def\NN{\mathbb{N}}

\def\RR{\mathbb{R}}

\def\SS{\mathbb{S}}
\def\ZZ{\mathbb{Z}}

\def\ds{\displaystyle}
\def\ni{\noindent}
\def\bs{\bigskip}
\def\ms{\medskip}

\def\eps{\varepsilon}

\def\fref#1{{\rm (\ref{#1})}}
\def\pref#1{{\rm \ref{#1}}}

\def\intquad {\int\mbox{\hspace{-3mm}}\int{\hspace{-3mm}}\int{\hspace{-3mm}}\int}

\def\calC{{\mathcal C}}

\def\calE{{\mathcal E}}

\def\calB{{\mathcal B}}

\def\GG{\mathbb G}
\def\lnar{\langle \na_\sigma\rangle}
\def\Vc{V_{\rm c}}
\catcode`@=11
\def\supess{\mathop{\operator@font Sup\,ess}}
\catcode`@=12
\def\un{{\mathbbmss{1}}}
\title{The Schr\"odinger-Poisson system on the sphere}
\author{Patrick G\'erard}
\address{Universit\'e Paris-Sud 11, Laboratoire de Math\'ematiques d'Orsay, CNRS, UMR 8628, et Institut Universitaire de France}
\email{patrick.gerard@math.u-psud.fr}
\author{Florian M\'ehats}
\address{IRMAR, Universit\'e Rennes 1, France}
\email{florian.mehats@univ-rennes1.fr}
\urladdr{http://perso.univ-rennes1.fr/florian.mehats/}
\dedicatory{In memory of Naoufel Ben Abdallah,  with our friendship}
\begin{document}

\begin{abstract}
We study the Schr\"odinger-Poisson system on the unit sphere $\SS^2$ of $\RR^3$, modeling the quantum transport of charged particles confined on a sphere by an external potential. Our first results concern the Cauchy problem for this system. We prove that this problem is regularly well-posed on every $H^s(\SS ^2)$ with $s>0$, and not uniformly well-posed on $L^2(\SS ^2)$. The proof of well-posedness relies on multilinear Strichartz estimates, the proof of ill-posedness relies on the construction of a counterexample which concentrates exponentially on a closed geodesic. In a second part of the paper, we prove that this model can be obtained as the limit of the three dimensional Schr\"odinger-Poisson system, singularly perturbed by an external potential that confines the particles in the vicinity of the sphere.
\end{abstract}


\maketitle


\section{Introduction and main results}
\sloppy 

Let $\SS^2\subset\RR^3$ be the unit sphere. For functions defined on $\SS^2$, one considers the operator $G$ defined by
\begin{equation}
  \label{G}
  G(f)(x)=\frac{1}{4\pi}\int_{\SS^2} \frac{1}{|x-y|}f(y)d\sigma(y),
\end{equation}
where $\sigma$ denote the surface measure on $\SS^2$ and $|\cdot|$ is the Euclidean norm on $\RR^3$.

In this paper, we are interested in the following Schr\"odinger-Poisson system on $\SS^2$:
\begin{equation}
  \label{sp}
  i\pa_t u +\Delta_\sigma u=G(|u|^2)u,\qquad u(t=0)=u_0.
\end{equation}
Here $\Delta_\sigma$ denotes the Laplace-Beltrami operator on $\SS^2$. This system models the transport of a gas of quantum charged particles confined on a surface, here the sphere, and interacting through the Poisson potential, which is the Coulombian interaction in the Hartree approximation. Our purpose in studying this ideal system is twofold. 

First, we are interested in understanding the wellposedness of the Cauchy problem for \fref{sp}: choice of the function space, local in time or global in time solutions, stability of the flow map, \ldots From this point of view, this paper comes in the continuity of several recent works concerning the nonlinear Schr\"odinger equation on Riemannian manifolds or in inhomogeneous media \cite{BGT1, BGT2, BGT3, BGT4, BGT5, gerard, GP}.

Second, we wish to justify this system for modeling a quantum gas via some asymptotic analysis, starting from a more conventional 3D Schr\"odinger-Poisson system with a singular perturbation which stands for a strong confinement potential. Strongly confined Schr\"odinger-Poisson systems have previously been studied in Euclidean spaces: in \cite{bmp, magnetic} for the confinement on a plane and in \cite{bcfm} for the confinement on an axis. The idea here is to investigate the influence of the geometry on the confinement procedure. The case of the sphere can be seen as a step before more general manifolds, which can be interesting in some applications in the field of nanoelectronics. Quantum dynamical systems confined on a surface have been studied previously in \cite{dc,fh,teufel} in linear situations. Starting from a similar scaling on the transverse Hamiltonian, these authors consider the linear Schr\"odinger equation with a confinement on a general surface and derive an effective Hamiltonian which locally depends on the curvature properties of the surface. Here our approach is mainly concentrated on understanding the nonlinear effects.

Our first result states that this problem is locally well-posed in $H^s$, $s>0$ and globally well-posed in the energy space:

\begin{theorem}
\label{theo1}
Let $s>0$ a real number. For every bounded subset $\mathcal B \subset H^s(\SS^2)$, there exists $T\in (0,+\infty]$ and a  subspace $X_T$ of $\calC((-T,T),H^s(\SS^2)$ such that, for $u_0\in \mathcal B$, the Cauchy problem \fref{sp} admits a unique solution $u\in X_T$. For all $0<T'<T$, the application $u_0\mapsto u\in C([-T',T'],H^s(\SS^2))$ is Lipschitz continuous on $\mathcal B$. Moreover, if $s\geq 1$ one can choose $T=+\infty$ and this global solution satisfies the following two conservation laws:
\begin{equation}
  \label{conslaw}
  \|u(t)\|_{L^2}=Q_0\,,\qquad \|\na_\sigma u(t)\|_{L^2}^2+\frac 1 2\int_{\SS^2}G(|u|^2)|u|^2d\sigma=E_0\,.
\end{equation}
\end{theorem}

Let us now deal with the limit case $s=0$: after Theorem \ref{theo1}, the question whether this system is well-posed on $L^2(\SS^2)$ is natural. In the case of a plane, the answer is positive thanks to Strichartz estimates in dimension 2. Our second result is that, in the case of the sphere, this system is not locally (uniformly) well-posed on $\SS^2$. The precise statement is as follows.
\begin{theorem}
\label{theo2}
For all ball $\calB$ of $L^2(\SS^2)$ centered on 0 and for all $T>0$, the application $u_0\mapsto u$ is {\em not} uniformly continuous from $\calB\cap H^1(\SS^2)$, endowed with the $L^2$ norm, into $C([-T,T],L^2(\SS^2))$.
\end{theorem}

Our third result concerns the asymptotic analysis. The starting model is the Schr\"odinger-Poisson system (or Hartree equation) in $\RR^3$ with a strong potential that confines the particles near the sphere. We consider the following system, on the unknown $u^\eps(t,x)$:
\begin{equation}
  \label{schrod}
  i\pa_t u^\eps=-\Delta u^\eps+\Vc^\eps u^\eps+V(|u^\eps|^2)u^\eps,\qquad u^\eps(t=0)=u_0^\eps\,,
\end{equation}
where $\Vc^\eps$ and $V(|u|^2)$ are respectively the applied confinement potential and the selfconsistent Poisson potential, given by
\begin{equation}
\label{potconf}
\Vc^\eps(x)=\frac{1}{\eps^2}\Vc\left(\frac{|x|-1}{\eps}\right)
\end{equation}
and
\begin{equation}
  \label{poisson}
  V(|u|^2)(x)=\frac{1}{4\pi}\int\frac{1}{|x-x'|}|u(x')|^2\,dx'\,.
\end{equation}
The parameter $\eps>0$ is the order of magnitude of the width of the confined gas, compared to a typical length, for instance the radius of the sphere. The square of $\eps$ also denotes the ratio between the kinetic energy of the particles and the confinement energy. One refers to \cite{bmp} and \cite{magnetic}, where similar singularly perturbed systems were derived and put in dimensionless form, starting from systems in physical variables. Our aim is to derive a simpler asymptotic system satisfied by $u^\eps$ as $\eps\to 0$.

The function $\Vc$ is fixed, independent of $\eps$, and is supposed to go to infinity at the infinity, faster than the harmonic potential, as  stated in the following assumption. Note that we do not solve the asymptotic problem in the precise case of a harmonic confinement potential, due to a lack of a priori estimates (Assumption $\alpha>2$ is used several times in the proofs of Lemma \ref{noyaupoisson} and Proposition \ref{prop2}).
\begin{assumption}
\label{H1}
The confinement potential is a $\calC^\infty$ function such that
\be
\label{superharm}
\forall z\in \RR\qquad \Vc(z)\geq C|z|^\alpha,
\ee
where $\alpha>2$. Moreover, it satisfies the following condition:
\begin{equation}
\label{reinfc}
\forall k \in \NN, \quad \exists C_k>0:\,\forall z\in \RR
\qquad
\left|\frac{\pa^k \Vc}{\pa z^k}(z)\right|
\leq C_k \Vc(z).
\end{equation}
\end{assumption}

Let us define the energy adapted to our system. We set
$$\calB^1=\left\{u\in \dot{H^1}(\RR^3):\;(\Vc^\eps)^{1/2}u\in L^2(\RR^3)\right\},$$
equipped with the norm
\begin{equation}
  \label{B1}
\|u\|_{\calB^1}^2=\|\na u\|_{L^2}^2+\|(\Vc^\eps)^{1/2}u\|_{L^2}^2.
\end{equation}
The two conservations laws associated to \fref{schrod} are the conservation of mass and the conservation of energy:
\be
\label{consL2}
\|u^\eps(t)\|_{L^2}^2=\|u^\eps_0\|_{L^2}^2
\ee
\bee
\|\na u^\eps(t)\|_{L^2}^2+\|(\Vc^\eps)^{1/2}u^\eps(t)\|_{L^2}^2+\frac{1}{2}\|\na V(|u^\eps(t)|^2)\|_{L^2}^2\qquad\qquad &&\nonumber\\=\|\na u_0^\eps\|_{L^2}^2+\|(\Vc^\eps)^{1/2}u_0^\eps\|_{L^2}^2+\frac{1}{2}\|\na V(|u_0^\eps|^2)\|_{L^2}^2.&&\label{consenergy}
\eee
The scaling \fref{potconf} of the confinement potential suggests that, if $\|u_0^\eps\|_{L^2}^2$ is of order 1, it will be natural to start with an energy of order $\frac{1}{\eps^2}$. Consequently, the solution of \fref{schrod} will satisfy the following natural bounds:
$$\|u^\eps\|_{L^2}\leq C,\qquad \eps\|u^\eps\|_{\calB^1}\leq C.$$
From Theorem \ref{theo2}, one can guess that these bounds will not be sufficient in order to analyze the limit of $u$ as $\eps$ goes to zero. Consequently, we shall assume some regularity of the initial data with respect to the angular variable $\sigma$. To be more precise, let us introduce the spherical coordinates $(r,\sigma)\in \RR_+^*\times \SS^2$ defined for all $x\in \RR^3\setminus\{0\}$ by $r=|x|$ and $\sigma=\frac{x}{|x|}$. We recall that the Laplace operator has the following form in spherical coordinates:
$$\Delta=\frac{1}{r^2}\pa_r\left(r^2\pa_r\right)+\frac{1}{r^2}\Delta_\sigma.$$
In particular, a remarkable property is that this operator commutes with $\Delta_\sigma$. This property will be crucial is our nonlinear analysis. In the sequel, we shall set $\lnar=(1-\Delta_\sigma)^{1/2}$. We introduce a last notation. The transversal confinement operator is the following operator:
\be
\label{Hr}
H_r=-\frac{1}{r^2}\pa_r\left(r^2\pa_r\right)+\Vc^\eps(r),
\ee
with domain 
$$D(H_r)=\left\{u\in L^2(\RR_+,r^2dr):\,\frac{1}{r^2}\pa_r\left(r^2\pa_ru\right)\in L^2(\RR_+,r^2dr),\, \Vc^\eps u\in L^2(\RR_+,r^2dr)\right\}.$$ Assumption \ref{H1} implies that $H_r$ is self-adjoint on $L^2(\RR_+,r^2dr)$, with a compact resolvent. We shall denote by $(E_p^\eps)_{p\in \NN}$, $(\psi_p^\eps)_{p\in \NN}$ its eigenvalues and its eigenfunctions.

We make the following assumption on the initial data.
\begin{assumption}
\label{Hinit}
The initial data $u_0^\eps$ satisfies
\be
\label{init1}
\|\lnar u_0^\eps\|_{L^2}+\eps\|u_0^\eps\|_{\calB^1}\leq C,
\ee
where $C>0$ is independent of $\eps\in (0,1]$.
Moreover, we have
\be
\label{init2}
\limsup_{\eps\to 0}\left(\|\un_{\eps^2 H_r>R} \,\lnar u^\eps_0\|_{L^2}+\| \un_{\lnar>R} \,\lnar u^\eps_0\|_{L^2}\right)\to 0\mbox{ as }R\to +\infty.
\ee
\end{assumption}
Let us comment on these assumptions. In \fref{init1}, the $L^2$ bound of $\lnar u_0$ means a supplementary decay at the infinity for this quantity, compared to a $H^1$ norm which would only control the $L^2$ norm of $\frac{1}{r}\lnar u_0^\eps$. The second assumption means that $\lnar u_0^\eps$ is partially relatively compact in $L^2$ with respect to the $\sigma$ variable $\eps^2$-oscillatory with respect to the operator $H_r$ defined by \fref{Hr} (see \cite{GG} for an abstract definition of $\eps$-oscillatory sequences with respect to a self-adjoint operator).

Our third theorem is the following one.
\begin{theorem}
\label{thm3}
Under Assumptions \pref{H1} and \pref{Hinit} on the confinement potential and the initial data, for all $\eps>0$, \fref{schrod} admits a unique global solution $u^\eps\in \calC(\RR,L^2\cap \calB_1)$. Moreover, the following equation admits a unique global solution $v^\eps\in \calC^0(\RR,\calB^1)$ such that $\lnar v^\eps\in  \calC^0(\RR,L^2)$:
\begin{equation}
  \label{schrodlim}
  i\pa_tv^\eps=H_rv^\eps-\Delta_\sigma v^\eps+\GG\left(|v^\eps|^2\right)v^\eps,\qquad v^\eps(t=0)=u_0^\eps\,,
\end{equation}
where $\GG$ is the following generalization of the operator $G$ defined in \fref{G}: for functions defined on $\RR^3$,
\begin{equation}
  \label{G2}
  \GG(f)(r\sigma)=\frac{1}{4\pi}\int_0^{+\infty}\int_{\SS^2} \frac{1}{|\sigma-\sigma'|}f(r'\sigma') r'^2dr'd\sigma'.
\end{equation}
Finally, we have the approximation of \fref{schrod} by the limit system \fref{schrodlim}: for all $T>0$,
$$\lim_{\eps\to 0}\|\lnar (u^\eps-v^\eps)\|_{L^\infty([-T,T],L^2)}=0.$$
\end{theorem}
Notice that, in \fref{schrodlim}, the nonlinear Schr\"odinger dynamics operates only in the $\sigma$ variable. Indeed, the function $\omega^\eps$ defined by
$$\omega^\eps(t,r,\sigma)=e^{itH_r}v^\eps(t,r,\sigma)$$
solves the system
\begin{equation}
  \label{schrodlim2}
  i\pa_t\omega^\eps=-\Delta_\sigma \omega^\eps+\GG\left(|\omega^\eps|^2\right)\omega^\eps,\qquad \omega^\eps(t=0)=u_0^\eps\,,
\end{equation}
which is a "mixed-state" version of the scalar equation \fref{sp}. Indeed, let us decompose the function $\omega^\eps$ on the eigenvectors of the confinement operator $H_r$, setting
$$\omega^\eps(t,r,\sigma)=\sum_{p=0}^{+\infty}\omega_p^\eps(t,\sigma)\psi_p^\eps(r),$$
i.e.
$$v^\eps(t,r,\sigma)=\sum_{p=0}^{+\infty}e^{-itE_p^\eps}\omega_p^\eps(t,\sigma)\psi_p^\eps(r).$$
Then each component $\omega^\eps_p$ satisfies
$$
  i\pa_t\omega_p^\eps=-\Delta_\sigma \omega_p^\eps+G\left(\sum_{q=0}^{+\infty}|\omega_q^\eps|^2\right)\omega_p^\eps\,,\qquad \omega_p^\eps(t=0)=\int_0^{+\infty}u_0\,\psi_p^\eps\, r^2dr\,.
$$
Notice that the components $\omega_p^\eps$ are only coupled through the selfconsistent potential. In particular, if the initial data is polarized on a single eigenmode, i.e. $u_0^\eps=v_{p_0}\psi_{p_0}^\eps$, then the solution of \fref{schrodlim} remains polarized on the same mode $p_0$, for all time: $v^\eps(t,r,\sigma)=e^{-itE^\eps_{p_0}/\eps^2}\omega^\eps_{p_0}(t,\sigma)\psi_{p_0}^\eps(r)$ and $\omega_{p_0}^\eps$ satisfies \fref{sp}.

\bs

\section{The Schr\"odinger-Poisson system on the sphere}

This section is devoted to the analysis of the Schr\"odinger-Poisson system \fref{sp} on the two-dimensional sphere $\SS^2$. As announced in the introduction, we will prove that this system is locally well-posed in $H^s(\SS^2)$ for all $s>0$, but not for the critical case $L^2(\SS^2)$. To make precise this statement, let us recall the notion of well-posedness that is meant here.
\begin{definition}
\label{def-wellposedness}
Let $s\in \RR$. We shall say that the equation \fref{sp} is, locally in time, uniformy well-posed in $H^s(\SS^2)$ if, for any bounded subset $B$ of $H^s(\SS^2)$, there exists $T>0$ and a Banach space $X_T$ continuously embedded into $C([-T,T],H^s(\SS^2))$, such that
\begin{itemize}
\item[(i)] For every Cauchy data $u_0\in B$, \fref{sp} admits a unique solution $u\in X_T$\,,
\item[(ii)] If $u_0\in H^\sigma(\SS^2)$ for $\sigma>s$, then $u\in C([-T,T],H^\sigma(\SS^2))$,
\item[(iii)] The map $u_0\in B\mapsto u\in X_T$ is uniformly continuous.
\end{itemize}
\end{definition}

The proof of Theorem \ref{theo1} uses essentially the techniques developed in the works of Burq, G\'erard, Tzvetkov \cite{BGT1,BGT2,BGT3,BGT4}. It can be decomposed into two main steps: the proof of well-posedness under the assumption of a multilinear estimate and the proof of the multilinear estimate itself. More precisely, we follow the work of G\'erard and Pierfelice \cite{GP}, who have adapted the methodology to the particular structure of Hartree-type nonlinearities as in \fref{sp}. The analysis relies in a crucial way on a quadrilinear estimate. For the sake of completeness and readability, we sketch all the steps of this proof, although the new contribution in this proof essentially concerns the quadrilinear estimate proved in subsection \ref{sect-quadri}.

More precisely, Theorem \ref{theo1} can be seen as a direct corollary of the following Propositions \ref{theo-quadri} and \ref{prop-quadri}, completed with the conservation laws \fref{conslaw} (which are very classical are will not be proved here).

\subsection{Well-posedness via quadrilinear estimate}
\label{sect-bourgain}

The following proposition reduces the study of wellposedness for \fref{sp} to a quadrilinear estimate.

\begin{proposition}
\label{theo-quadri}
 Suppose that, for every $\chi\in \calC^\infty_0(\RR)$ and for all $\eps\in(0,\frac{1}{2})$, there exists $C_\eps>0$ such that, for any $f_1$, $f_2$, $f_3$, $f_4 \in L^2(\SS^2)$ satisfying
  \begin{equation}
    \label{local}
    \un_{\sqrt{1-\Delta_\sigma}\in [N_j, 2N_j]}(f_j)=f_j,\qquad j=1,2,3,4,
  \end{equation}
one has the following quadrilinear estimate:
\begin{equation}
  \label{quadri}
\ds \sup_{\tau\in \RR}\left|\int_\RR \int_{\SS^2}\chi(t)e^{it\tau}G(u_1u_2)u_3u_4\,d\sigma dt\right| \leq C_\eps (m(N_1,N_2,N_3,N_4))^\eps\prod_{i=1}^4 \|f_j\|_{L^2(\SS^2)}
\end{equation}
where $u_j(t)=S(t) f_j$ for $j=1,2,3,4$, $S(t)=e^{it\Delta_\sigma}$ denoting the free evolution, and where $m(N_1,N_2,N_3,N_4)$ denotes the product of the smallest two numbers among $N_1$, $N_2$, $N_3$, $N_4$. Then the Cauchy problem \fref{sp} is locally uniformly well-posed in $H^s(\SS^2)$ for any $s>0$.
\end{proposition}
\begin{proof}
We follow the steps of the proof of Theorem 1 in \cite{BGT2}, see also \cite{BGT4,GP}. Proposition \ref{theo-quadri} will be proved thanks to a fixed point procedure formulated in the Bourgain spaces $X^{s,b}$.

\bs
\ni
{\em Step 1. Reformulation of the problem in the Bourgain spaces}.

\ms
\ni
Following the definitions in \cite{bourgain} and \cite{BGT4}, we introduce the following family of Hilbert spaces
\begin{equation}
\label{xsb}
X^{s,b}=\left\{u\in {\mathcal S}'(\RR\times \SS^2)\,:\,\,\left(1+|i\pa_t+\Delta_\sigma|^2\right)^{b/2}(1-\Delta_\sigma)^{s/2}u\in L^2(\RR\times \SS^2)\right\},
\end{equation}
for $s,b\in \RR$. For any $T>0$, we also denote by $X^{s,b}_T$ the space of restrictions of elements of $X^{s,b}$ to $(-T,T)\times \SS^2$. We gather in the following lemma some interesting properties of these spaces.
\begin{lemma}
\label{lem-bourgain}
The Bourgain space $X^{s,b}$ satisfies the following properties.
\begin{enumerate}
\item $\forall f\in H^s(\SS^2)$, $\forall b>0$, the function $S(t)f$ belongs to $X^{s,b}$.
\item $\forall b>\frac{1}{2}$, $\quad X^{s,b}\hookrightarrow {\calC}(\RR,H^s(\SS^2))$.
\item Let $b,b'$ such that $0\leq b'<\frac{1}{2}$, $0\leq b<1-b'$. There exists $C>0$ such that, if $T\in (0,1]$ and $w(t)=\int_0^tS(t-\tau)f(\tau)d\tau$, then
$$\|w\|_{X^{s,b}_T}\leq CT^{1-b-b'}\|f\|_{X^{s,-b'}_T}\,.$$
\end{enumerate}
\end{lemma}
\begin{proof}
The first property (i) stems directly from the definition \fref{xsb} of $X^{s,b}$.
Remarking that 
\begin{equation}
\label{sobo}
\|u\|_{X^{s,b}}=\|v\|_{H^b(\RR, H^s(\SS^2))}, \qquad \mbox{where }v(t)=S(-t)u(t),\end{equation}
the second statement (ii) is obvious and the last statement (iii) appears to be a consequence of the following elementary result, proved e.g. in \cite{ginibre}:
$$\mbox{if }g(t)=\int_0^th(\tau)d\tau\quad \mbox{then}\quad \|g\|_{H^b(-T,T)}\leq CT^{1-b-b'}\|f\|_{H^{-b'}(-T,T)}.$$
\end{proof}
As a consequence of this lemma, it is easy to prove by a standard contraction argument that the Cauchy problem \fref{sp} is locally uniformly well-posed in $H^s(\SS^2)$, $s>0$, as soon as the following result holds true:
\begin{lemma}
\label{claim}
Assume that the quadrilinear estimate of Proposition \pref{theo-quadri} holds true. Then for all $s>0$, one can find $b,b'$ satisfying $0<b'<1/2<b<1-b'$ and such that we have the estimate
\begin{equation}
  \label{cl1uuu}
  \|G(u_1u_2)u_3\|_{X^{s,-b'}}\leq C\|u_1\|_{X^{s,b}}\|u_2\|_{X^{s,b}}\|u_3\|_{X^{s,b}},
\end{equation}
\begin{equation}
  \label{cl2}
  \|G(|u|^2)u\|_{X^{\sigma,-b'}}\leq C\|u\|_{X^{s,b}}^2\|u\|_{X^{\sigma,b}}\quad \mbox{for }\sigma\geq s.
\end{equation}
\end{lemma}
The end of this section consists in the proof of this lemma.

\bs
\ni
{\em Step 2. Quadrilinear estimate in $X^{s,b}$}.

\ms
\ni
Let us first give without proof two estimates on the operator $G$. These results can be obtained straightforwardly using Hardy-Littlewood-Sobolev inequalities in dimension 2.
\begin{lemma}
\label{lemmeG}
The operator $G$ defined by \fref{G}satisfies the following estimates:
\begin{equation}
  \label{es1}
  \left\|G(f)\right\|_{L^{q}(\SS^2)}\leq C\,\|f\|_{L^p(\SS^2)},\qquad \mbox{for }1<p<2\mbox{ and }\frac{1}{q}=\frac{1}{p}-\frac{1}{2},
\end{equation}
\begin{equation}
  \label{es2}
  \left\|G(f)\right\|_{L^\infty(\SS^2)}\leq C\,\|f\|_{L^p(\SS^2)}^\theta\,\|f\|_{L^1(\SS^2)}^{1-\theta},\qquad \mbox{for }p>2\mbox{ and }\theta=\frac{p}{2p-2}.
\end{equation}
\end{lemma}
We now reformulate the quadrilinear estimate \fref{quadri} in the $X^{s,b}$ spaces.
\begin{lemma}
\label{reform}
Under the same assumption as in Proposition \pref{theo-quadri}, let $u_1$, $u_2$, $u_3$, $u_4$ satisfy
 \begin{equation}
    \label{local2}
    \un_{\sqrt{1-\Delta_\sigma}\in [N_j, 2N_j]}(u_j)=u_j,\qquad j=1,2,3,4.
  \end{equation}
Then, for every $s>0$, there exists $0<b'<1/2$ such that
\begin{equation}
  \label{quadri2}
\left|\int_\RR \int_{\SS^2} G(u_1u_2)u_3u_4\,d\sigma dt\right| \leq C\,m(N_1,N_2,N_3,N_4)^s\prod_{i=1}^4 \|u_j\|_{X^{0,b'}}.
\end{equation}
\end{lemma}
\begin{proof}
The desired estimate \fref{quadri2} can be obtained by interpolation between the two following inequalities:
\begin{equation}
\label{first}
\left|\int_\RR \int_{\SS^2} G(u_1u_2)u_3u_4\,d\sigma dt\right| \leq C\,m(N_1,N_2,N_3,N_4)^{1/2}\prod_{i=1}^4 \|u_j\|_{X^{0,1/4}}
\end{equation}
and, for any $b>1/2$ and $0<\eps<1/2$,
\begin{equation}
\label{second}
\left|\int_\RR \int_{\SS^2} G(u_1u_2)u_3u_4\,d\sigma dt\right| \leq C_\eps\,m(N_1,N_2,N_3,N_4)^\eps\prod_{i=1}^4 \|u_j\|_{X^{0,b}}\,.
\end{equation}

Let us prove the first estimate \fref{first}.
By symmetry and self-adjointness of $G$, only two cases have to be considered: $m(N_1,N_2,N_3,N_4)=N_1N_2$ and $m(N_1,N_2,N_3,N_4)=N_1N_3$. In the first case, we deduce from H\"older and from \fref{es2} that
$$\begin{array}{ll}
\ds \left|\int_\RR \int_{\SS^2} G(u_1u_2)u_3u_4\,d\sigma dt\right| \leq \|G(u_1u_2)\|_{L^2(\RR,L^\infty(\SS^2))}\|u_3u_4\|_{L^2(\RR,L^1(\SS^2))}\\[3mm]
\ds \qquad \leq C\|u_1u_2\|_{L^2(\RR,L^4(\SS^2))}^{2/3}\|u_1u_2\|_{L^2(\RR,L^1(\SS^2))}^{1/3}\|u_3u_4\|_{L^2(\RR,L^1(\SS^2))}\\[3mm]
\ds \qquad \leq C N_1^{1/2}N_2^{1/2}\prod_{i=1}^4 \|u_j\|_{L^4(\RR,L^2(\SS^2))}\,,
\end{array}
$$
since $\|u_i\|_{L^8(\SS^2)}\leq CN_i^{3/4}\|u_i\|_{L^2(\SS^2)}$ due to the spectral localization. Hence, to get \fref{first} in this case, it suffices to remark that
$$\|u_j\|_{L^4(\RR,L^2(\SS^2))}=\|S(-t)u_j\|_{L^4(\RR,L^2(\SS^2))}\leq C\|S(-t)u_j\|_{H^{1/4}(\RR,L^2(\SS^2))}=C\|u_j\|_{X^{0,1/4}},$$
where we used the isometric action of $S(t)$,  a Sobolev embedding in dimension one and \fref{sobo}.

The second case $m(N_1,N_2,N_3,N_4)=N_1N_3$ can be treated as follows. Using \fref{es1}, we have
$$\begin{array}{ll}
\ds \left|\int_\RR \int_{\SS^2} G(u_1u_2)u_3u_4\,d\sigma dt\right| \leq \|G(u_1u_2)\|_{L^2(\RR,L^4(\SS^2))}\|u_3u_4\|_{L^2(\RR,L^{4/3}(\SS^2))}\\[3mm]
\ds \qquad \leq C\|u_1u_2\|_{L^2(\RR,L^{4/3}(\SS^2))}\|u_3u_4\|_{L^2(\RR,L^{4/3}(\SS^2))}\\[3mm]
\ds \qquad \leq C\|u_1\|_{L^4(\RR,L^4(\SS^2))}\|u_2\|_{L^4(\RR,L^2(\SS^2))}\|u_3\|_{L^4(\RR,L^4(\SS^2))}\|u_4\|_{L^4(\RR,L^2(\SS^2))}\\[3mm]
\ds \qquad \leq C N_1^{1/2}N_3^{1/2}\prod_{i=1}^4 \|u_j\|_{L^4(\RR,L^2(\SS^2))}\leq C N_1^{1/2}N_2^{1/2}\prod_{i=1}^4 \|u_j\|_{X^{0,1/4}}\,,
\end{array}
$$
where we used $\|u_i\|_{L^4(\SS^2)}\leq CN_i^{1/2}\|u_i\|_{L^2(\SS^2)}$.

We now sketch the proof of the second estimate \fref{second}, for all $b>1/2$, which is directly inspired from \cite{BGT4,GP}. Let us start by proving in a first step the result for the special case where the $u_j$ are supported in time in an interval of size 1, say $(0,1)$ after translation. Let $v_j(t)=S(-t)u_j(t)$ and let $f_j(\tau)=\hat{v}_j(\tau)$ denote the Fourier transform in time of $v_j$. Applying the inverse Fourier transform, we get
$$\begin{array}{l}
\ds \int_{\SS^2} G(u_1u_2)u_3u_4\,d\sigma=\frac{1}{(2\pi)^2}\intquad_{\RR^4}e^{it(\tau_1+\tau_2+\tau_3+\tau_4)}\times\\[3mm]
\ds \qquad \times G(S(t)f_1(\tau_1)S(t)f_2(\tau_2))S(t)f_3(\tau_3)S(t)f_4(\tau_4)d\sigma d\tau_1d\tau_2d\tau_3d\tau_4.
\end{array}$$
Hence, the quadrilinear estimate \fref{quadri}, applied pointwise in $\tau_1$, $\tau_2$, $\tau_3$, $\tau_4$ and after having chosen $\chi\in {\mathcal C}^\infty_0(\RR)$ such that $\chi=1$ on $(0,1)$, gives
$$\begin{array}{ll}
\ds \left|\int_\RR \int_{\SS^2} G(u_1u_2)u_3u_4\,d\sigma dt\right|
&\ds \leq C_\eps\,m(N_1,N_2,N_3,N_4)^\eps \prod_{i=1}^4 \int_\RR \|\hat{v}_j\|_{L^2(\SS^2)}(\tau)d\tau\\
&\ds \leq C_\eps\,m(N_1,N_2,N_3,N_4)^\eps\prod_{i=1}^4 \|\langle \tau\rangle^b\hat{v}_j\|_{L^2(\RR\times\SS^2)}\\
&\ds \quad =C_\eps\,m(N_1,N_2,N_3,N_4)^\eps \prod_{i=1}^4 \|u_j\|_{X^{0,b}}\,,
\end{array}
$$
where the Cauchy-Schwarz inequality could be applied since $b>1/2$. 

To treat the general case, it suffices now to introduce a function $\psi\in {\mathcal C}^\infty_0(\RR)$ supported in $(0,1)$ such that $\sum_{n\in \ZZ}\psi(t-\frac{n}{2})=1$, to decompose $u_j=\sum_{n\in \ZZ}u_{j,n}$, with $u_{j,n}=\psi(t-\frac{n}{2})u_j(t)$, $j=1,2,3,4$, and to apply the result of the first step to the functions $u_{j,n}$. The conclusion follows from the almost orthogonality property satisfied by the Bourgain spaces
$$\sum_{n\in\ZZ}\|u_{j,n}\|_{X^{0,b}}^2\leq C\|u_{j}\|_{X^{0,b}}^2\,.$$
\end{proof}

\bs
\ni
{\em Step 3. Proof of Lemma \pref{claim} by dyadic decomposition}. 
\nopagebreak

\ms
\ni
We have now the tools to prove Lemma \ref{claim}. We shall only prove \fref{cl1uuu}, the proof of \fref{cl2} being similar. By duality, we have to show that, for all $u_4\in X^{-s,b'}$,
\begin{equation}
  \label{duality}
  \left|\int_\RR\int_{\SS^2} G(u_1u_2)u_3u_4d\sigma dt\right|\leq C\left(\prod_{j=1}^3\|u_j\|_{X^{s,b}}\right)\|u_4\|_{X^{-s,b'}}.
\end{equation}
Let us perform a dyadic decomposition of the $u_j$'s, setting $u_{j,N}=\un_{\sqrt{1-\Delta_\sigma}\in [N,2N]}(u_j)$. We have
\begin{equation}
  \label{eq}
\|u_j\|_{X^{s,b}}^2\simeq \sum_{N}N^{2s}\|u_{j,N}\|_{X^{0,b}}^2\simeq \sum_{N}\|u_{j,N}\|_{X^{s,b}}^2
\end{equation}
where $N$ are dyadic integers. Proving \fref{duality} is equivalent to estimating the sum
$$\sum_{N_1,N_2,N_3,N_4}J(N_1,N_2,N_3,N_4),$$
where
$$J(N_1,N_2,N_3,N_4)=\int_\RR\int_{\SS^2} G(u_{1,N_1}u_{2,N_2})u_{3,N_3}u_{4,N_4}d\sigma dt.$$

We now remark that in the above sum, all the terms with $N_4>4\max(N_1,N_2,N_3)$ are zero. To prove this point, we first claim that the operator $G$ is a function of $\Delta_\sigma$, more precisely, that we have for all $f$
\begin{equation}
  \label{equivG}
  G(f)=\frac{1}{4\pi}\int_{\SS^2} \frac{1}{|x-y|}f(y)d\sigma(y)=(1-4\Delta_\sigma)^{-1/2}(f).
\end{equation}
The proof of the claim is done below. Moreover, from spectral localization  $u_{j,N_j}$ is the sum of spherical harmonics of degrees between $\sqrt{N_j/2}$ and $\sqrt{2N_j}$. One can deduce from properties of products of spherical harmonics that the product $G(u_{1,N_1}u_{2,N_2})u_{3,N_3}$ is a sum of spherical harmonics of degrees less than $\sqrt{2\max(N_1,N_2,N_3)}$, so orthogonal to $u_4$ if $N_4>4\max(N_1,N_2,N_3)$.

Hence, by symmetry, it suffices to consider the terms with $N_1\leq N_2\leq N_3$ and $N_4\leq 4N_3$. Pick $s'$ such that $0<s'<s$. By Lemma \ref{reform}, there exists $0<b'<1/2$ such that
$$\begin{array}{l}
\ds \sum_{N_1,N_2,N_3,N_4}J(N_1,N_2,N_3,N_4)\leq C\sum_{N_1,N_2,N_3}\sum_{N_4\leq 4N_3} N_1^{s'}N_2^{s'}\prod_{i=1}^4 \|u_{j,N_j}\|_{X^{0,b'}} ,\\
\ds \quad \leq C\left(\prod_{j=1,2}\sum_{N_j}N_j^{s'-s}\|u_{j,N_j}\|_{X^{s,b'}}\right)\sum_{N_3}\sum_{N_4\leq 4N_3} \left(\frac{N_3}{N_4}\right)^{-s}\|u_{3,N_3}\|_{X^{s,b'}}\|u_{4,N_4}\|_{X^{-s,b'}} 
\end{array}$$
where we used \fref{eq}. The series in $N_1$ and $N_2$ converges easily. Denoting $N_4=2^n$ and $N_3/N_4=2^m$, the series in $N_3$, $N_4$ can be bounded by a simple Cauchy-Schwarz argument. Indeed, this series reads
$$\begin{array}{l}
\ds \sum_{m \geq -2}2^{-ms}\sum_{n\geq \max(0,-m)}\|u_{3,2^{m+n}}\|_{X^{s,b'}}\|u_{4,2^n}\|_{X^{-s,b'}}\\
\ds \qquad \qquad \leq
\left(\sum_{m \geq -2}2^{-ms}\right)\left(\sum_{n}\|u_{3,2^n}\|_{X^{s,b'}}^2\right)^{1/2}\left(\sum_{n}\|u_{4,2^n}\|_{X^{-s,b'}}^2\right)^{1/2}\\[6mm]
\ds \qquad \qquad \leq C \|u_3\|_{X^{s,b'}}\|u_4\|_{X^{-s,b'}}\,.
\end{array}$$
The proof of Lemma \ref{claim} is complete, so Proposition \ref{theo-quadri} is proved.

\ms
\ni
{\em Proof of the claim \fref{equivG}.} We use a very old argument that can be found, e.g., in Poincar\'e's works \cite{poincare}. Let us compute explicitely the action of $G$ on a spherical harmonic $Y^m_\ell$. We define the following function on $\RR^3$:
$$u(r,\sigma)=r^\ell Y^m_\ell(\sigma)\mbox{ for }r\leq 1,\qquad u(r,\sigma)=r^{-\ell-1}Y^m_\ell(\sigma)\mbox{ for }r\geq 1.$$
It is readily seen that $u$ is harmonic on $\RR^3\setminus\SS^2$, so by computing explicitely the jump of $\pa_r u$ on $\SS^2$ we obtain
$$-\Delta u=(2\ell+1)Y^m_\ell d\sigma$$
in the distribution sense, where $d\sigma$ is the surface measure on $\SS^2$. Since $u$ is decreasing at the infinity, this means that
$$u(r,\sigma)=\frac{1}{4\pi}\int_{\SS^2}\frac{1}{|r\sigma-\sigma'|}(2\ell+1)Y^m_\ell d\sigma',$$
thus
$$u(1,\sigma)=(2\ell+1)G(Y^m_\ell).$$
Denoting $\lambda=\ell(\ell+1)$, we have $2\ell+1=\sqrt{1+4\lambda}$, so we have
$$G(Y^m_\ell)=\frac{1}{\sqrt{1+4\lambda}}Y^m_\ell$$
and \fref{equivG} is proved.
\end{proof}

\subsection{Quadrilinear estimate}
\label{sect-quadri}
In this section, we prove the quadrilinear estimate \fref{quadri} which is, thanks to Proposition \ref{theo-quadri}, the core of the proof of local existence result in Theorem \ref{theo1}. The main result of this section is the following proposition.
\begin{proposition}
\label{prop-quadri}
There exists a constant $C>0$ such that for any $f_1$, $f_2$, $f_3$, $f_4 \in L^2(\SS^2)$ satisfying
  \begin{equation}
    \label{local3}
    \un_{\sqrt{1-\Delta_\sigma}\in [N_j, 2N_j]}(f_j)=f_j,\qquad j=1,2,3,4,
  \end{equation}
and for all $\eps\in ]0,\frac{1}{2}[$ one has the following quadrilinear estimate
\begin{equation}
  \label{quadri3}
\ds  \sup_{\tau\in \RR}\left|\int_\RR \int_{\SS^2}\chi(t)e^{it\tau}G(u_1u_2)u_3u_4\,d\sigma dt\right| \leq C_\eps\,(m(N_1,N_2,N_3,N_4))^\eps\prod_{i=1}^4 \|f_j\|_{L^2(\SS^2)}
\end{equation}
where $u_j(t)=S(t) f_j$ for $j=1,2,3,4$, $S(t)=e^{it\Delta_\sigma}$ denoting the free evolution, and $\chi\in \calC^\infty_0(\RR)$ is arbitrary. 
\end{proposition}
\begin{proof}
Let us decompose the $f_j$'s on spherical harmonics:
$$f_j=\sum_{n_j=N_j/2}^{2N_j}H_{n_j}^j\qquad \mbox{with}\qquad \Delta_\sigma H_{n_j}^j+n_j(n_j+1) H_{n_j}^j=0,$$
so that
$$u_j(t)=\sum_{n_j=N_j/2}^{2N_j}e^{-itn_j(n_j+1)} H_{n_j}^j\,.$$
The quantity to estimate can be written as
$$\begin{array}{l}
\ds \int_\RR \int_{\SS^2}\chi(t)e^{it\tau}G(u_1u_2)u_3u_4\,d\sigma dt=\sum_{n_1,n_2,n_3,n_4}\omega_{n_1,n_2,n_3,n_4}(\tau)I(n_1,n_2,n_3,n_4)
\end{array}
$$
where
$$\omega_{n_1,n_2,n_3,n_4}(\tau)=\hat{\chi}\left(\tau-\sum_{j=1}^4\eps_j n_j(n_j+1)\right)\quad \mbox{with }\eps_j=(-1)^{j+1},$$
$\hat{\chi}$ being the Fourier transform of $\chi$,
and
$$I(n_1,n_2,n_3,n_4)=\int_{\SS^2}G(H_{n_1}^1H_{n_2}^2)H_{n_3}^3H_{n_4}^4d\sigma\,.$$

By symmetry and self-adjointness of $G$, assume for instance that $N_3=\min(N_1,N_2,N_3,N_4)$. Let us estimate $I(n_1,n_2,n_3,n_4)$ by adapting the proof of multilinear estimates in \cite{BGT2,BGT3,BGT4,GP}. 
One can decompose $H_{n_j}^j$ using a microlocal partition of the unity with semiclassical cut-off of the form $\chi_j(x,h_jD)$ (with small support):
$$H_{n_j}^j=\varphi_j+\sum_k\chi_j^k(x,h_jD)H_{n_j}^j,$$
where $\varphi_j$ is regular (for all $s$, $\||D^s|\varphi_j\|_{L^2}\leq C_s \|H_{n_j}^j\|_{L^2}$) and where the sum is finite.
Thus, we only have to estimate some terms of the form
$$\tilde I=\int_{\SS^2}G(V_1V_2)V_3V_4d\sigma,$$
where $V_j$ is of the form $V_j=\chi_j(x,h_jD)H_{n_j}^j$.
The key point, from Burq, G\'erard, Tzvetkov \cite{BGT2,BGT3,BGT4} is that, using semiclassical Strichartz estimates in dimension one, for each choice of $(\chi_j)_{j=1,2,3,4}$ one can find a local system of linear coordinates $(x,y)$ such that, for all $p,q\in [2,\infty]$ satisfying $\frac{2}{p}+\frac{1}{q}=\frac{1}{2}$, we have
\be
\label{stri}
\|V_j\|_{L^p_xL^q_y}\leq CN_j^{1/p}\|H_{n_j}^j\|_{L^2(\SS^2)}, \qquad \mbox{for }j=1,2,3,4.
\ee
The proof of this result uses the fact that the function $H_{n_j}^j$ satisfies
$$h_j^2 \Delta_\sigma H_{n_j}^j+H_{n_j}^j=0, \qquad \mbox{with }h_j^2=\frac{1}{n_j(n_j+1)}\sim N_j^2,$$
which can be locally seen as an evolution equation where the well-chosen variable $x$ plays the role of a time.

In local coordinates we have
$$|\tilde I|\leq C\|G(V_1V_2)\|_{L^\infty_{x}L^{s}_{y}}\|V_3\|_{L^{1/\eps}_xL^{r}_y}\|V_4\|_{L^2_xL^2_y}$$
with $s=\frac{1}{2\eps}$ and $\frac{1}{r}=\frac{1}{2}-2\eps$. Since
$$|G(V_1V_2)|\leq C\int\frac{1}{|x-x'|+|y-y'|}\,|V_1V_2|(x',y')dx'dy',$$
a Hardy-Littlewood Sobolev inequality gives
$$\begin{array}{ll}
\ds \|G(V_1V_2)\|_{L^\infty_{x}L^{s}_{y}}
&\ds \leq C\sup_x\int \frac{1}{|x-x'|^{1-2\eps}}\,\|V_1V_2(x',\cdot)\|_{L^1}\,dx'\\[4mm]
&\ds \leq C_\eps \|V_1\|_{L^\infty_xL^2_y}\|V_2\|_{L^\infty_xL^2_y},
\end{array}$$
where we used the fact that all the functions are compactly supported. Then we apply \fref{stri} with the pairs $(p,q)=(\infty,2)$ and $(\frac{1}{\eps},r)$:
$$\begin{array}{ll}
\ds |\tilde I|&\ds \leq C_\eps \|V_1\|_{L^\infty_xL^2_y}\|V_2|_{L^\infty_xL^2_y}\|V_3\|_{L^{1/\eps}_xL^{r}_y}\|V_4\|_{L^2_xL^2_y}\\
&\ds  \leq C_\eps (N_3)^\eps \prod_{j=1}^4\|H_{n_j}^j\|_{L^2}.
\end{array}$$
To conclude, it remains to estimate uniformly with respect to $\tau$ the quantity
$$Q(\tau)=\sum_{n_1,n_2,n_3,n_4}|\omega_{n_1,n_2,n_3,n_4}(\tau)| \prod_{j=1}^4\|H_{n_j}^j\|_{L^2}$$
where, in the sum, $n_1,n_2,n_3,n_4$ satisfy
\begin{equation}
  \label{condnj}
  \forall j\in \{1,2,3,4\}\qquad \frac{N_j}{2}\leq n_j\leq 2N_j\,.
\end{equation}
This quantity has already been estimated in \cite{GP}, but we reproduce again this proof here for the sake of completeness.

Denote by $\Lambda(k)$ the set of $(n_1,n_2,n_3,n_4)$ satisfying \fref{condnj} and
$$\sum_{j=1}^4 \eps_j n_j(n_j+1)=k.$$
We deduce from the decay of $\hat{\chi}$ at the infinity that
$$\begin{array}{ll}
\ds Q(\tau)&\ds \leq C\sum_{\ell \in \ZZ}\frac{1}{1+\ell^2}\sum_{\Lambda([\tau]+\ell)}\prod_{j=1}^4\|H_{n_j}^j\|_{L^2}\\[5mm]
&\ds \leq C\sup_{k\in\ZZ}\sum_{\Lambda(k)}\prod_{j=1}^4\|H_{n_j}^j\|_{L^2}\,.
\end{array}$$
Let us denote by $\alpha,\beta$ the two indices such that $m(N_1,N_2,N_3,N_4)=(N_\alpha,N_\beta)$ and by $\gamma,\delta$ the two other indices: $\{\gamma,\delta\}=\{1,2,3,4\}\setminus\{\alpha,\beta\}$. Then we introduce the set
$$\begin{array}{r}
\ds \Gamma(k,i,j)=\left\{(n_i,n_j)\,: \;\frac{N_i}{2}\leq n_i\leq 2N_i,\;\frac{N_j}{2}\leq n_j\leq 2N_j \right.\qquad \qquad \\
\ds \left.\mbox{ and }\,\eps_in_i(n_i+1)+\eps_jn_j(n_j+1)=k\vphantom{\frac{N}{2}}\right\}.
\end{array}
$$
Now, denoting
$$S(k,i,j)=\sum_{\Gamma(k,i,j)}\|H_{n_i}^i\|_{L^2}\|H_{n_j}^j\|_{L^2},$$
we split the sum as follows:
$$Q(\tau)\leq  C \sup_{k\in\ZZ}\sum_{k'\in \ZZ}S(k,\alpha,\gamma)S(k-k',\beta,\delta).$$
Then we apply the following elementary result of number theory (see e.g. \cite{BGT2}).
\begin{lemma}
  Let $\sigma=\pm 1$. For all $\eps>0$, there exists $C_\eps>0$ such that, given $M\in \ZZ$ and $N\in\NN^*$,
$$\#\left\{(k_1,k_2)\in \NN^2:\;\;N\leq k_1\leq 2N,\;\;k_1^2+\sigma k_2^2=M\right\}\leq C_\eps N^\eps.$$
\end{lemma}
Hence we get
$$\sup_{k'}\# \Gamma(k',\alpha,\gamma)\leq C_\eps N_\alpha^\eps\,,\qquad \sup_{k,k'}\# \Gamma(k-k',\beta,\delta)\leq C_\eps N_\beta^\eps\,,$$
and by Cauchy-Schwarz
$$\begin{array}{l}
\ds \sum_{k'\in \ZZ}S(k,\alpha,\gamma)S(k-k',\beta,\delta)\leq C_\eps (N_\alpha N_\beta)^\eps\times\\
\ds \qquad \times \left(\sum_{k'}\sum_{\Gamma(k',\alpha,\gamma)}\|H_{n_\alpha}^\alpha\|_{L^2}^2\|H_{n_\gamma}^\gamma\|_{L^2}^2\right)^{1/2} \left(\sum_{k'}\sum_{\Gamma(k',\beta,\delta)}\|H_{n_\beta}^\beta\|_{L^2}^2\|H_{n_\delta}^\delta\|_{L^2}^2\right)^{1/2}\\
\ds \qquad \leq C_\eps (N_\alpha N_\beta)^\eps \prod_{j=1}^4\|f_j\|_{L^2},
\end{array}
$$
where we used the orthogonality of the spherical harmonics. The proof is complete.
\end{proof}

\subsection{An instability result in $L^2$}
In this section, we prove Theorem \ref{theo2} and exhibit a high frequency instability for the Cauchy problem in $L^2(\SS^2)$ for the Schr\"odinger-Poisson system \fref{sp}. Consider the following spherical harmonic on $\SS^2$
\begin{equation}
\label{psi}
\psi_n(x)=(x_1+ix_2)^n
\end{equation}
where $(x_1,x_2,x_3)$ are cartesian coordinates on $\RR^3$, $\SS^2=\{x_1^2+x_2^2+x_3^2=1\}$. Remark that this function concentrates on the equator $\{x_3=0\}$ as $n\to \infty$.

In the case of the cubic nonlinear Schr\"odinger equation on $\SS^d$, an instability in some $H^s$ space has been shown in \cite{BGT5} (see also \cite{banica} for a more precise result) by finding an ansatz for the solution of the equation with an initial data proportional to $\psi_n$ and of order 1 in $H^s$. A different approach has been presented in \cite{gerard} in order to exhibit the same instability. It consists in constructing a stationary solution for the nonlinear Schr\"odinger equation, rapidly oscillating in time, by minimizing the nonlinear energy of the problem. The point is to estimate precisely the corresponding nonlinear eigenvalue as $n$ goes to infinity. Let us adapt this argument to our case.

\begin{lemma}
\label{lemtech}
The function $\psi_n$ defined by \fref{psi} satisfies
\begin{equation}
  \label{estim1}
\|\psi_n\|_{L^2}^2=\frac{C_1}{\sqrt{n}}+{\mathcal O}(n^{-3/2}),\qquad
\|\na_\sigma\psi_n\|_{L^2}^2=n(n+1)\|\psi_n\|_{L^2}^2,\quad
\end{equation}
\begin{equation}
  \label{estim2}
  \int_{\SS^2} G(|\psi_n|^2)|\psi_n|^2\,d\sigma= C_2\frac{\log n}{n}+\frac{C_3}{n}+{\mathcal O}\left(\frac{1}{n^2}\right),
\end{equation}
where $C_1$, $C_2$, $C_3$ are constant real numbers.
\end{lemma}
\begin{proof}
The first estimates \fref{estim1} are immediate. Let us prove \fref{estim2}, writing
$$\begin{array}{l}
\ds \int_{\SS^2} G(|\psi_n|^2)|\psi_n|^2\,d\sigma=\int\frac{(\cos\theta)^{2n+1}(\cos\theta')^{2n+1}}{\ds(|\cos\theta e^{i\varphi}-\cos\theta' e^{i\varphi'}|^2+(\sin\theta-\sin\theta')^2)^{1/2}}d\theta d\theta'd\varphi d\varphi'\\[5mm]
\ds =\int(\cos\theta)^{2n+1/2}(\cos\theta')^{2n+1/2}\left(\frac{1-\cos(\theta-\theta')}{\cos\theta\cos\theta'}+1-\cos(\varphi-\varphi')\right)^{-1/2}d\theta d\theta'd\varphi d\varphi',
\end{array}$$
where the integration domain is $(\theta,\theta',\varphi,\varphi')\in [-\pi/2,\pi/2]\times  [-\pi/2,\pi/2]\times[0,2\pi]\times[0,2\pi]$.
As $n\to +\infty$, the main contribution in this integral is near the equator $\theta=0$, $\theta'=0$. Hence some elementary analysis using
$$\int_0^{2\pi}\int_0^{2\pi}\left(t+1-\cos(\varphi-\varphi')\right)^{-1/2}d\varphi d\varphi'=a\log t+b+{\mathcal O}(t)\quad \mbox{as }t\to 0+,$$
yields \fref{estim2}.
\end{proof}

Now, for all $0<\delta\leq 1$ we introduce the function $\phi_n=c_n\psi_n$ such that $\|\phi_n\|_{L^2}=\delta$. We deduce from Lemma \ref{lemtech} that
\begin{equation}
  \label{estiphi1}
  \|\na_\sigma\phi_n\|_{L^2}^2=n(n+1)\delta^2,
\end{equation}
\begin{equation}
  \label{estiphi2}
\int_{\SS^2} G(|\phi_n|^2)|\phi_n|^2\,d\sigma=C_4 \delta^4\log n+C_5\delta^4+{\mathcal O}\left(\frac{\delta^4}{n}\right),
\end{equation}
for some constants $C_4$ and $C_5$ independent of $n$ and $\delta$. We will see that this factor $\log n$ is at the origin of the high frequency instability.

Let us define as in \cite{gerard} the space
$$L^2_n=\left\{f\in L^2(\SS^2):\,\forall \alpha\in \RR\quad f\circ R_\alpha=e^{in\alpha}f\right\},$$
where $R_\alpha$ denotes the rotation of angle $\alpha$ around the $x_3$ axis. Since the spherical harmonics take the form
$$Y^m_\ell=c_{m,\ell}\,P^m_\ell(\sin \theta)e^{im\varphi},$$
in spherical coordinates ($\theta$ is the angle between the vector $x$ and the equatorial plane $\{x_3=0\}$), where $P^m_\ell$ is a Legendre function and $c_{m,\ell}$ is a normalization factor, it is readily seen that the space $L^2_n$ is characterized by
\begin{equation}
  \label{L2n}
  L^2_n=\mbox{span}\left\{Y^{n}_{n+k}\,,\; k\in \NN\right\},
\end{equation}
where $n$ is fixed.
Therefore, we have a characterization of $\phi_n$ (up to a phase factor) as the minimizer of the Dirichlet energy on $L^2_n$:
$$\phi_n=\mbox{argmin}\left\{\|\nabla_\sigma u\|_{L^2}^2\,\mbox{for }u\in L^2_n \mbox{ such that }\|u\|_{L^2}=\delta\right\}.$$
Let us now minimize the energy associated to \fref{sp}:
$$\calE(u)=\|\na_\sigma u\|_{L^2}^2+\frac{1}{2}\int_{\SS^2} G(|u|^2)|u|^2\,d\sigma
$$
on the sphere of radius $\delta$ of $L^2_n$. It is easy to prove the compactness of minimizing sequences and to obtain the existence of a minimizer $f_n$ to this problem. The Euler equation reads
$$-\Delta_\sigma f_n+G(|f_n|^2)f_n=\omega_nf_n\,,$$
with
\begin{equation}
  \label{omegan}
  \omega_n=\frac{1}{\delta^2}\left(\|\na_\sigma f_n\|_{L^2}^2+\int_{\SS^2} G(|f_n|^2)|f_n|^2\,d\sigma\right)>0,
\end{equation}
and the function 
$$u_n(t,x)=e^{-it\omega_n}f_n(x)$$
is a solution of \fref{sp}. The key of the method is now to show that $f_n$ is close to $\phi_n$, up to a phase factor, and to give a precise estimate for $\omega_n$.
\begin{lemma}
\label{lemfn}
There exists a constant $C$ independent of $n\in \NN^*$ and $\delta\in(0,1]$ such that, for some $\alpha_n\in \RR$,
\begin{equation}
  \label{estifn}
  \left\|f_n-e^{i\alpha_n}\phi_n\right\|_{L^2}\leq C\delta^2\left(\frac{\log n}{n}\right)^{1/2},
\end{equation}
\begin{equation}
  \label{estiomegan}
  \omega_n=n(n+1)+C_4\delta^2\log n+C_5\delta^2+{\mathcal O}\left(\delta^3\frac{\log n}{n^{1/4}}\right).
\end{equation}
\end{lemma}
\begin{proof}
Since $f_n$ belongs to $L^2_n$, let us decompose this function on the spherical harmonics, according to \fref{L2n}:
\begin{equation}
\label{decomp}
f_n=a_0\phi_n+\sum_{k=1}^\infty a_kY_{n+k}^n\,,
\end{equation}
the $Y_\ell^n$ being chosen normalized in $L^2$.
One can deduce from the normalization conditions that
\begin{equation}
  \label{n1}
|a_0|^2\delta^2+\sum_{k=1}^\infty |a_k|^2=\delta^2
\end{equation}
and the property $\calE(f_n)\leq \calE(\phi_n)$ reads
\begin{equation}
  \label{n2}
\begin{array}{r}
\ds  |a_0|^2\delta^2n(n+1)+ \sum_{k=1}^\infty|a_k|^2(n+k)(n+k+1)+\frac{1}{2}\int_{\SS^2} G(|f_n|^2)|f_n|^2\,d\sigma\\[3mm]
\ds \qquad \leq \delta^2n(n+1)+\frac{1}{2}\int_{\SS^2} G(|\phi_n|^2)|\phi_n|^2\,d\sigma\,.
\end{array}
\end{equation}
By substracting $n(n+1)\times$\fref{n1} to \fref{n2}, we obtain
\begin{equation}
  \label{n3bis}
  \sum_{k=1}^\infty k(2n+k+1)|a_k|^2\leq \frac{1}{2}\int_{\SS^2} G(|\phi_n|^2)|\phi_n|^2\,d\sigma\leq C\delta^4\log n,
\end{equation}
where we used \fref{estiphi2}. Hence we have
  \begin{equation}
  \label{n3}
 \|f_n-a_0\phi_n\|_{L^2}^2=\sum_{k=1}^\infty |a_k|^2\leq C\delta^4\frac{\log n}{n}.
\end{equation}
Inserting \fref{n3} in \fref{n1} leads to
\begin{equation}
  \label{n4}
  0\leq 1-|a_0|^2\leq C\delta^2\frac{\log n}{n}
\end{equation}
and \fref{estifn} follows, setting $e^{i\alpha_n}=\frac{a_0}{|a_0|}$. 

Let us now prove \fref{estiomegan}. By combining the two inequalities
$$\|\na_\sigma f_n\|_{L^2}^2+\frac{1}{2}\int_{\SS^2} G(|f_n|^2)|f_n|^2\,d\sigma\leq \|\na_\sigma \phi_n\|_{L^2}^2+\frac{1}{2}\int_{\SS^2} G(|\phi_n|^2)|\phi_n|^2\,d\sigma$$
and
$$\|\na_\sigma \phi_n\|_{L^2}^2\leq \|\na_\sigma f_n\|_{L^2}^2,$$
and by using \fref{omegan} it comes
\begin{equation}
\label{dif}
\begin{array}{r}
\ds 0\leq \|\na_\sigma \phi_n\|_{L^2}^2+\int_{\SS^2} G(|\phi_n|^2)|\phi_n|^2\,d\sigma-\delta^2\omega_n\qquad \qquad \qquad \qquad \\
\ds\leq \int_{\SS^2} \left(G(|\phi_n|^2)|\phi_n|^2\,d\sigma-G(|f_n|^2)|f_n|^2\right)d\sigma.
\end{array}
\end{equation}
Moreover, for all $M>0$ and for all $u,v$ in the centered ball of $L^{8/3}(\SS^2)$ of radius $M$, we deduce from the estimate \fref{es1} on the operator $G$ that
\begin{equation}
  \label{cl1}
  \left|\int_{\SS^2} \left(G(|u|^2)|u|^2-G(|v|^2)|v|^2\right)d\sigma\right|\leq CM^3\|u-v\|_{L^{8/3}}.
\end{equation}
Let us now  estimate $f_n$, $\phi_n$ and the difference between these two functions in $L^{8/3}$. 
The key point will be the following estimate on spherical harmonics due to Sogge \cite{sogge}:
\begin{equation}
  \label{sogge}
\|Y_\ell^n\|_{L^{8/3}}\leq C\ell^{1/16}\|Y_\ell^n\|_{L^2}.
\end{equation}
Therefore we have
\begin{equation}
\label{phi8/3}
\|\phi_n\|_{L^{8/3}}\leq C\delta n^{1/16}.
\end{equation}
Moreover, from the decomposition \fref{decomp} of $f_n$ and by the Minkowski inequality,
$$\|f_n-e^{i\alpha_n}\phi_n\|_{L^{8/3}}\leq \left\|\left(a_0-\frac{a_0}{|a_0|}\right)\phi_n\right\|_{L^{8/3}}+\sum_{k=1}^\infty |a_k|\|Y_{n+k}^n\|_{L^{8/3}}.$$
By \fref{sogge}, \fref{n3bis} and Cauchy-Schwarz, we have
$$
\begin{array}{ll}
\ds \sum_{k=1}^\infty |a_k|\|Y_{n+k}^n\|_{L^{8/3}}
&\ds \leq C\left(\sum_{k=1}^\infty \frac{(n+k)^{1/8}}{k(2n+k+1)}\right)^{1/2}\left(\sum_{k=1}^\infty k(2n+k+1)|a_k|^2\right)^{1/2}\\[5mm]
&\ds \leq C\delta^2\frac{\log n}{n^{7/16}}.
\end{array}$$
Hence, using also \fref{n4} and \fref{phi8/3}, we obtain
$$\|f_n-e^{i\alpha_n}\phi_n\|_{L^{8/3}}\leq C\delta^2\frac{\log n}{n^{7/16}}\quad \mbox{and}\quad \|f_n\|_{L^{8/3}}\leq C\delta n^{1/16}.$$
Finally, \fref{cl1} leads to
$$\left|\int_{\SS^2} \left(G(|\phi_n|^2)|\phi_n|^2-G(|f_n|^2)|f_n|^2\right)d\sigma\right|\leq C\delta^5\frac{\log n}{n^{1/4}}.$$
By inserting this estimate in \fref{dif}, then by using \fref{estiphi1}, \fref{estiphi2}, we obtain \fref{estiomegan}.
\end{proof}

\ms
\ni
{\bf \em Proof of Theorem \pref{theo2}.}

\ms
\ni
We now have the tools to prove the high frequency instability and Theorem \ref{theo2}. Remark first that, replacing $f_n$ by $f_ne^{-i\alpha_n}$, one may assume $\alpha_n=0$ in \fref{estifn}. As in \cite{gerard}, we consider two values of $\delta$:
$$\delta_n=\delta_0,\qquad \delta'_n=\kappa_n\delta_0,$$
where $0<\delta_0\leq 1$ and $\kappa_n\to 1$ as $n\to +\infty$ in a way that is defined below (see \fref{kappan}), and we denote by $f_n$, $\phi_n$ and $f'_n$, $\phi'_n$ the corresponding functions.

By \fref{estifn}, we have
$$\|f_n-f'_n\|_{L^2}\leq \|f_n-\phi_n\|_{L^2}+\|f'_n-\phi'_n\|_{L^2}+|\delta_n-\delta'_n|\leq C\left(\frac{\log n}{n}\right)^{1/2}\delta_0+|1-\kappa_n|\delta_0,$$
whereas the corresponding solutions of \fref{sp}
$$u_n=e^{-it\omega_n}f_n,\qquad u'_n=e^{-it\omega'_n}f'_n$$
satisfy
$$\begin{array}{ll}
\ds \|u_n-u'_n\|_{L^2}&\ds \geq \left|e^{-it\omega_n}-e^{-it\omega'_n}\right|-\|f_n-f'_n\|_{L^2}\\
&\ds \geq \left|\sin\left(\frac{t}{2}(\omega_n-\omega'_n)\right)\right|\delta_0-C\left(\frac{\log n}{n}\right)^{1/2}\delta_0-|1-\kappa_n|\delta_0.\end{array}$$
By \fref{estiomegan}, we have
$$\omega_n-\omega'_n=C_4(1-\kappa_n^2)\delta_0^2\log n +C_5(1-\kappa_n^2)\delta_0^2+{\mathcal O}\left(\frac{\log n}{n^{1/4}}\right).$$
If now we define $\kappa_n$ by
\begin{equation}
  \label{kappan}
  \kappa_n=\left(1-(\log n)^{-1/2}\right)^{1/2}
\end{equation}
and an observation time $t_n$ by
$$t_n =\frac{\pi}{\omega_n-\omega'_n}=\frac{\pi}{C_4\delta_0^2(\log n)^{1/2}}(1+o(1)),$$
we have clearly a sequence $t_n\to 0+$ such that
$$\|(u_n-u'_n)(t_n,\cdot)\|_{L^2}\geq \delta_0-\eps_n,\qquad \|(u_n-u'_n)(0,\cdot)\|_{L^2}\leq \eps_n,$$
where $\eps_n\to 0$ as $n\to +\infty$. Since $u_n$ and $u'_n$ are in the ball of radius $\delta_0$ of $L^2(\SS^2)$, this contradicts Item (iii) of Definition \ref{def-wellposedness}, which means that \fref{sp} is not uniformly well-posed on $L^2(\SS^2)$.
\qed

\bs

\bs
\section{Asymptotic analysis}
\label{sect3}

The aim of this second part of the paper is to prove Theorem \ref{thm3}, which justifies the use of \fref{sp} (in fact, of the mixed-state version \fref{schrodlim} of this equation) as a model for quantum transport on a surface by means of asymptotic analysis,  deriving this system from a well-established model, the three-dimensional Schr\"odinger-Poisson system \fref{schrod}. The Cauchy problem for this system without the confinement potential was studied in \cite{bm,ilz} in the energy space, and in \cite{castella} in $L^2$.

\subsection{Estimating the Poisson nonlinearity}
In this subsection, we obtain some estimates on the Poisson nonlinearity for functions confined near the sphere by the confinement operator $\Vc^\eps$. We will see that the following family of norms is well-adapted for the study of the singularly perturbed nonlinear problem \fref{schrod}:
$$\|\lnar^{s} u\|_{L^2}+\eps\|u\|_{\calB^1},$$
where $s$ is a positive integer, whereas the following family of norms is well-adapted to the limit problem \fref{schrodlim}:
$$\|\lnar^{s_1} u\|_{L^2}+\|(1+\eps^2 H_r)^{s_2/2}u\|_{L^2},$$
where $s_1$ and $s_2$ are nonnegative integers and where we recall the definition \fref{Hr} of the Hamiltonian $H_r$, which is nonnegative thanks to Assumption \ref{H1}. The nonlinear analysis of \fref{schrod} will be based on Lemmas \ref{noyaupoisson} and \ref{noyaupoisson2}.
\begin{lemma}
\label{noyaupoisson}
Let $u\in \calB^1$, such that $\lnar^s u\in L^2$ with $s\geq 1$. Then there exists $\gamma>0$ and a constant $C>0$ such that, for all $\eps\in (0,1]$,  the nonlinearities $V(|u|^2)u$ and $\GG(|u|^2)u$, respectively defined by \fref{poisson} and \fref{G2}, satisfy the tame estimate
\bea
\left\|\lnar^s\left(\GG(|u|^2)u\right)\right\|_{L^2}&\leq& C\|\lnar u\|_{L^2}^2\|\lnar^s u\|_{L^2}\label{pois2}\\
\left\|\lnar^s\left(V(|u|^2)u\right)\right\|_{L^2}&\leq& C\left(\|\lnar u\|_{L^2}^2+\eps^{2(1+\gamma)}\,\|u\|_{\calB^1}^2\right)\|\lnar^s u\|_{L^2}\label{pois1}
\eea
where $C$ is independent of $\eps$. Moreover, if $u,v \in \calB^1$ are such that $\lnar u,\,\lnar v\in L^2$, then
\be
\label{poisuv2}
\begin{array}{l}
\ds \left\|\lnar\left(\GG(|u|^2)u-\GG(|v|^2)v\right)\right\|_{L^2}\\[3mm]
\ds \qquad \leq C\left(\|\lnar u\|_{L^2}^2+\|\lnar v\|_{L^2}^2\right)\|\lnar (u-v)\|_{L^2},
\end{array}
\ee
\be
\label{poisuv1}
\begin{array}{l}
\ds\left\|\lnar\left(V(|u|^2)u-V(|v|^2)v\right)\right\|_{L^2}\\[3mm]
\ds \qquad \leq C\left(\|\lnar u\|_{L^2}^2+\eps^{2(1+\gamma)}\,\|u\|_{\calB^1}^2+\|\lnar v\|_{L^2}^2+\eps^{2(1+\gamma)}\,\|v\|_{\calB^1}^2\right)\times\\[3mm]
\ds \hspace*{4.5cm} \times \left(\|\lnar (u-v)\|_{L^2}+\eps^{2(1+\gamma)}\,\|u-v\|_{\calB^1}\right)\,.\end{array}
\ee
Finally, if $\lnar^2 u\in L^2$, then we have
\be
\left\|\lnar\left(V(|u|^2)u-\GG(|u|^2)u\right)\right\|_{L^2}\leq C\eps^\gamma\left(\|\lnar^2 u\|_{L^2}^3+\eps^3\,\|u\|_{\calB^1}^3\right)
.\label{poisdiff}
\ee
\end{lemma}
\begin{proof} 

\bs
\ni
{\em Step 1. Sobolev embeddings}. Let us first write some anisotropic Sobolev embeddings that will be useful several times in the proof. In spherical coordinates (recall that $x=r\sigma$), we shall denote
$$\forall p,q \in [1,+\infty)\qquad \|u\|_{L^p_r L^q_\sigma}=\left(\int_{0}^\infty\left(\int_{\SS^2} |u|^q\, d\sigma\right)^{q/p}r^2dr\right)^{1/q},$$
$$\forall q \in [1,+\infty)\qquad \|u\|_{L^\infty_r L^q_\sigma}=\supess_{r>0} \left\|u(r,\cdot)\right\|_{L^q_\sigma}.$$
By Sobolev embeddings in dimension 2, we have
\be
\label{em1}
\left\|u\right\|_{L^2_rL^p_\sigma}\leq C_p \|\lnar u\|_{L^2}
\ee
for all $p\in[2,\infty)$. Moreover, since the $H^1$ norm reads in spherical coordinates
$$\|u\|_{H^1}^2=\|u\|_{L^2}^2+\|\pa_ru\|^2_{L^2}+\left\|\frac{1}{r}\na_\sigma u\right\|^2_{L^2},$$
we deduce from the Hardy inequality
$$\left\|\frac{u}{r}\right\|_{L^2}\leq C\|\na u\|_{L^2}$$
that
\be
\label{bornehardy}
\left\|\lnar\left(\frac{u}{r}\right)\right\|_{L^2}\leq C\|u\|_{H^1}.
\ee
Hence, by Sobolev embeddings, we get
\be
\label{em2}
\left\|\frac{u}{r}\right\|_{L^2_rL^p_\sigma}\leq C_p \|u\|_{H^1}\leq C_p \|u\|_{\calB^1}
\ee
for all $p\in[2,\infty)$. Finally, we claim that
\be
\label{em3}
\left\|\frac{u}{r^{1/2}}\right\|_{L^2_rL^4_\sigma}\leq C\|\lnar u\|_{L^2}+C\eps^{1+\gamma}\|u\|_{\calB^1},
\ee
for some $\gamma>0$ independent of $\eps$.
In order to prove \fref{em3}, let us split the integral on $\{r> 1/2\}\cup\{r< 1/2\}$. We get
\bee
&&\left\|\frac{u}{r^{1/2}}\right\|_{L^2_rL^4_\sigma}= \left\|\frac{u}{r^{1/2}}\un_{r>1/2}\right\|_{L^2_rL^4_\sigma}+\left\|\frac{u}{r^{1/2}}\un_{r<1/2}\right\|_{L^2_rL^4_\sigma}\\
&&\qquad\leq\sqrt{2}\left\|u\un_{r>1/2}\right\|_{L^2_rL^4_\sigma}+C\left\|u\un_{r<1/2}\right\|_{L^2}^{1/2-\eta}\left\|u\un_{r<1/2}\right\|_{L^2_rL^{p_1}_\sigma}^{\eta}\left\|\frac{u}{r}\un_{r<1/2}\right\|_{L^2_rL^{p_1}_\sigma}^{1/2}\\
&&\qquad\leq C\|\lnar u\|_{L^2}+C\left\|u\un_{r<1/2}\right\|_{L^2}^{1/2-\eta}\|\lnar u\|_{L^2}^{\eta}\left\|u\right\|_{\calB^1}^{1/2}
\eee
where we used a H\"older inequality for the second inequality and \fref{em1}, \fref{em2} for the third one. Here $\eta\in(0,1/2)$ is a parameter that will be fixed later and $p_1=2+1/\eta$. It remains to use the properties of the confinement operator. From Assumption \ref{H1}, we deduce
\bea
\left\|u\un_{r<1/2}\right\|_{L^2}\leq2^{\alpha/2}\left\|(r-1)^{\alpha/2} u\right\|_{L^2}&\leq&C\eps^{\alpha/2} \left\|\Vc\left(\frac{r-1}{\eps}\right)^{1/2}u\right\|_{L^2}\nonumber\\
&\leq& C\eps^{1+\alpha/2}\|u\|_{\calB^1}.\label{confresu}
\eea
Hence,
\bee
\left\|\frac{u}{r^{1/2}}\right\|_{L^2_rL^4_\sigma}&\leq&C\|\lnar u\|_{L^2}+C\eps^{(1/2-\eta)(1+\alpha/2)}\left\|u\right\|_{\calB^1}^{1-\eta}\|\lnar u\|_{L^2}^{\eta}
\eee
To conclude, we choose $\eta$ small enough such that $(1/2-\eta)(1+\alpha/2)> 1-\eta$. This is possible thanks to the assumption $\alpha>2$. The Young inequality finally gives  \fref{em3}.

\bs
\ni
{\em Step 2. Proof of \fref{pois2}}. Let $s\geq 1$. From the Littlewood-Paley theory and the Mikhlin-H\"ormander multiplier theorem \cite{AG,schlag} applied on the sphere, we have
\bea
\left\|\lnar^s\left(\GG(|u|^2)u\right)\right\|_{L^2}&\leq &C\left\|\GG(|u|^2)\right\|_{L^\infty}\left\|\lnar^s u\right\|_{L^2}\nonumber\\&&+C\left\|\lnar^s \GG(|u|^2)\right\|_{L^p_\sigma}\|u\|_{L^2_rL^q_\sigma}\label{lp1}
\eea
for all pair $(p,q)$ such that $p>2$ and $\frac{1}{q}+\frac{1}{p}=\frac{1}{2}$.
The $L^\infty$ norm of $\GG(|u|^2)$ is easy to estimate. Indeed, this function is independent of the variable $r$ and we notice that $\GG$ is linked to the operator $G$ defined by \fref{G} thanks to the relation
$$\GG (|u|^2)=\int_0^{+\infty} G(|u|^2(r',\cdot))\,r'^2dr',$$
Pick a real number $p>2$. From \fref{es2} and the H\"older inequality, we deduce
\bea
0\leq \GG(|u|^2)&\leq& C\int_0^{+\infty} \left\|u(r'\cdot))\right\|_{L^{2p}_\sigma}^{2p/(2p-2)}\left\|u(r'\cdot))\right\|_{L^2_\sigma}^{(2p-4)/(2p-2)}\,r'^2dr'\nonumber\\
&\leq &C\left\|u\right\|_{L^2_rL^{2p}_\sigma}^{2p/(2p-2)}\left\|u\right\|_{L^2}^{(2p-4)/(2p-2)}\nonumber\\
&\leq& C\|\lnar u\|_{L^2}^2\label{norme1}
\eea
where we used the Sobolev embedding \fref{em1}.

Let us now estimate the $L^\infty_rL^p_\sigma$ norm of $\lnar^s \GG(|u|^2)$, where $p>2$ is given. Since, from \fref{equivG}, we have
$$\GG(|u|^2)=\int_0^{+\infty}(1-4\Delta_\sigma)^{-1/2}(|u|^2)r^2dr,$$
the operator $-\Delta_\sigma$ commutes with $\GG$ and therefore
$$\lnar^s\GG(|u|^2)=\GG\left(\lnar^s(|u|^2)\right).$$
Thus, by \fref{es1},
\bea
\left\|\lnar^s\GG(|u|^2)\right\|_{L^p_\sigma}&\leq& C\int_0^{+\infty}\left\|\lnar^s(|u|^2)(r\cdot)\right\|_{L^{p_2}_\sigma}r^2dr\nonumber\\
&\leq&C\int_0^{+\infty}\left\|\lnar^su(r\cdot)\right\|_{L^2_\sigma}\left\|u(r\cdot)\right\|_{L^p_\sigma}r^2dr\nonumber\\
&\leq &C\left\|\lnar^su\right\|_{L^2}\left\|u\right\|_{L^2_rL^p_\sigma}\label{norme2}
\eea
where $p_2$ was chosen such that $\frac{1}{p_2}=\frac{1}{p}+\frac{1}{2}$ and where the Mikhlin-H\"ormander multiplier theorem on the sphere was used again. Finally, from \fref{lp1}, \fref{norme1}, \fref{norme2} and \fref{em1}, we deduce \fref{pois2}.

\bs
\ni
{\em Step 3. Proof of \fref{pois1}}. Let $s\geq 1$. By the Mikhlin-H\"ormander multiplier theorem, we have
\bea
\left\|\lnar^s\left(V(|u|^2)u\right)\right\|_{L^2}&\leq & C\left\|V(|u|^2)\right\|_{L^\infty}\left\|\lnar^s u\right\|_{L^2}\nonumber\\&& +C\left\|r^{1/2}\lnar^s V(|u|^2)\right\|_{L^\infty_rL^4_\sigma }\left\|\frac{u}{r^{1/2}}\right\|_{L^2_rL^4_\sigma}\label{lp2}
\eea
Let us first estimate the $L^\infty$ norm of $V(|u|^2)$, written in spherical coodinates:
\begin{equation}
  \label{poissonsph}
  V(|u|^2)(r\sigma)=\frac{1}{4\pi}\int\frac{1}{|r\sigma-r'\sigma'|}|u(r'\sigma')|^2\,r'^2dr'd\sigma'\,.
\end{equation}
Since $\sigma$ and $\sigma'$ are unitary, we have $(\sigma-\sigma')\cdot(\sigma+\sigma')=0$, thus
\be
\label{identity}
|r\sigma-r'\sigma'|^2=\left(\frac{r+r'}{2}\right)^2\,|\sigma-\sigma'|^2
+ \left(\frac{r-r'}{2}\right)^2\,|\sigma+\sigma'|^2,
\ee
which yields
\begin{equation}
  \label{maj}
  \frac{1}{|r\sigma-r'\sigma'|}\leq \frac{2}{\max(r,r')}\,\frac{1}{|\sigma-\sigma'|}\,.
\end{equation}
This enables to estimate $V(|u|^2)$ by using again Lemma \ref{lemmeG}. For all $p>2$, we have
\bee
0\leq V(|u|^2)&\leq &\int_0^{+\infty} \frac{2}{r'}G(|u|^2(r',\cdot))\,r'^2dr'\\
&\leq &C\int_0^{+\infty} \frac{1}{r'}\left\|u(r',\cdot))\right\|_{L^{2p}_\sigma}^{2p/(2p-2)}\left\|u(r',\cdot))\right\|_{L^2_\sigma}^{(2p-4)/(2p-2)}\,r'^2dr'.
\eee
Hence, by splitting the integral on $r> 1/2$ and $r< 1/2$ and using the H\"older inequality,
\bee
0\leq V(|u|^2)&\leq &C\left\|u\un_{r>1/2}\right\|_{L^2_rL^{2p}_\sigma}^{2p/(2p-2)}\left\|u\un_{r>1/2}\right\|_{L^2}^{(2p-4)/(2p-2)}\\
&&+C\left\|\frac{u}{r}\un_{r<1/2}\right\|_{L^2_rL^{2p}_\sigma}\left\|u\un_{r<1/2}\right\|_{L^2_rL^{2p}_\sigma}^{2/(2p-2)}\left\|u\un_{r<1/2}\right\|_{L^2}^{(2p-4)/(2p-2)}\\
&\leq &C\|\lnar u\|_{L^2}^2+C\|u\|_{\calB^1}\left\|\lnar u\right\|_{L^2}^{2/(2p-2)}\left\|u\un_{r<1/2}\right\|_{L^2}^{(2p-4)/(2p-2)}
\eee
where we have used again the Sobolev embeddings \fref{em1} and \fref{em2}. Finally, \fref{confresu} yields
$$
0\leq V(|u|^2)\leq C\|\lnar u\|_{L^2}^2+C\eps^{(1+\alpha/2)(1-\frac{2}{2p-2})}\|u\|_{\calB^1}^{2-\frac{2}{2p-2}}\left\|\lnar u\right\|_{L^2}^{2/(2p-2)}.
$$
To conclude, we choose $p$ large enough such that $(1+\alpha/2)(1-\frac{2}{2p-2})>2-\frac{2}{2p-2}$. This is possible thanks to the assumption $\alpha>2$ and we obtain finally, using the Young inequality,
\be
\label{norme4}
0\leq V(|u|^2)\leq C\|\lnar u\|_{L^2}^2+C\eps^{2(1+\gamma)}\|u\|_{\calB^1}^2.
\ee

Let us now estimate the $L^\infty_rL^4_\sigma$ norm of $r^{1/2}\lnar^s V(|u|^2)$. Since $V(|u|^2)=(-\Delta)^{-1}(|u|^2),$
the operator $-\Delta_\sigma$ commutes with $V$ and we have
$$\lnar^s V(|u|^2)=V\left(\lnar^s(|u|^2)\right).$$
Therefore, we deduce from \fref{poissonsph}, from \fref{maj}, from \fref{es1}, and finally from the Mikhlin-H\"ormander theorem on the sphere that
\bea
\left\|r^{1/2}\lnar^sV(|u|^2)\right\|_{L^\infty_rL^4_\sigma}&\leq& C\int_0^{+\infty}\frac{1}{r^{1/2}}\left\|\lnar^s(|u|^2)(r,\cdot)\right\|_{L^{4/3}_\sigma}r^2dr\nonumber\\
&\leq&C\int_0^{+\infty}\left\|\lnar^su(r,\cdot)\right\|_{L^2_\sigma}\left\|\frac{u(r,\cdot)}{r^{1/2}}\right\|_{L^4_\sigma}r^2dr\nonumber\\
&\leq &C\left\|\lnar^su\right\|_{L^2}\left\|\frac{u}{r^{1/2}}\right\|_{L^2_rL^4_\sigma}.\label{norme3}
\eea
Finally, we deduce \fref{pois1} from \fref{lp2}, \fref{norme4}, \fref{norme3} and \fref{em3}. The proofs of \fref{poisuv2} and \fref{poisuv1} are very similar to the proofs of \fref{pois2} and \fref{pois1}, with $s=1$ and we leave the details to the reader.

\bs
\ni
{\em Step 4. Proof of \fref{poisdiff}}. In order to estimate the $L^2$ norm of $\lnar\left(V(|u|^2)u-\GG(|u|^2)u\right)$, let us introduce
$$\delta(r,r',\sigma,\sigma'):=\frac{1}{|\sigma-\sigma'|}-\frac{1}{|r\sigma-r'\sigma'|},$$
such that, for all function $w$, we have
\be
\label{ecritdelta}
(V(w)-\GG(w))(r\sigma)=-\int_0^{+\infty}\int_{\SS^2}\delta(r,r',\sigma,\sigma')w(r'\sigma')r'^2dr'd\sigma'.
\ee
From \fref{maj}, one deduces directly that
\be
\label{majord}
|\delta(r,r',\sigma,\sigma')|\leq \frac{1}{|\sigma-\sigma'|}\left(1+\frac{2}{\max(r,r')}\right).
\ee
Moreover, by using \fref{identity}, one can decompose $\delta(r,r',\sigma,\sigma')=\delta_1+\delta_2$
with
$$\delta_1=\frac{((r+r')^2-4)|\sigma-\sigma'|^2}{4|r\sigma-r'\sigma'||\sigma-\sigma'|(|r\sigma-r'\sigma'|+|\sigma-\sigma'|)}$$
and
$$\delta_2=\frac{(r-r')^2|\sigma+\sigma'|^2}{4|r\sigma-r'\sigma'||\sigma-\sigma'|(|r\sigma-r'\sigma'|+|\sigma-\sigma'|)}.$$
Let $\chi$ be a continuous  function on $\RR $, positive for $z\ne 0$, such that $\chi(z)\sim |z|$ as $z\to 0$ and $\chi(z)=1$ for $|z|\geq 2$. From \fref{maj}, one deduces that
$$|\delta_1|\leq C\frac{(\chi(r-1)+\chi(r'-1))}{|\sigma-\sigma'|}\left(1+\frac{1}{\max(r,r')}\right).$$
Furthermore, since by \fref{identity} we have
$$|r-r'||\sigma+\sigma'|\leq 2|r\sigma-r'\sigma'|,$$
we can estimate $\delta_2$ as follows:
$$
|\delta_2|\leq \frac{|r-r'||\sigma+\sigma'|}{2|r\sigma-r'\sigma'||\sigma-\sigma'|}\leq C\frac{(\chi(r-1)+\chi(r'-1))}{|\sigma-\sigma'|^2}\left(1+\frac{1}{\max(r,r')}\right).
$$
We have thus proved that
$$|\delta(r,r',\sigma,\sigma')|\leq C\frac{(\chi(r-1)+\chi(r'-1))}{|\sigma-\sigma'|^2}\left(1+\frac{1}{\max(r,r')}\right).$$
By interpolating between this inequality and \fref{majord}, one gets finally, for all $\eta\in(0,1),$
\be
\label{estidelta}
\left|\delta(r,r',\sigma,\sigma')\right|\leq C\frac{\kappa(r,r')}{|\sigma-\sigma'|^{1+\eta}},
\ee
where we have set
$$\kappa(r,r')=\left((\chi(r-1))^\eta+(\chi(r'-1))^\eta\right)\left(1+\frac{1}{\max(r,r')}\right).$$

Next, we claim that one can adapt Steps 2 and 3 in order to prove that, for $\eta>0$ small enough,
\bea
&&\left\|\lnar\left(V(|u|^2)u-\GG(|u|^2)u\right)\right\|_{L^2}\nonumber\\
&&\qquad \leq C\|\chi(r-1)) \lnar u\|_{L^2}^\eta \|\lnar u\|_{L^2}^{3-\eta}+C\eps ^\gamma\left(\|\lnar u\|_{L^2}^3+\eps^3\|u\|_{\calB^1}^3\right).\qquad \qquad \label{chi}
\eea
Let us prove this claim. We have
\bea
&&\left\|\lnar\left(V(|u|^2)u-\GG(|u|^2)u\right)\right\|_{L^2}\nonumber\\
&&\qquad   \leq  C\left(\int_0^{+\infty}\left\|V(|u|^2)-\GG(|u|^2)\right\|_{L^\infty_\sigma}^2\left\|\lnar u\right\|_{L^2_\sigma}^2r^2dr\right)^{1/2}\nonumber\\&&\qquad \quad  +C\left(\int_0^{+\infty}\left\|V(\lnar(|u|^2))-\GG(\lnar(|u|^2))\right\|_{L^4_\sigma}^2\left\|u\right\|_{L^4_\sigma}^2r^2dr\right)^{1/2}.\label{mik25}
\eea
By \fref{ecritdelta} and \fref{estidelta}, we have
\bee
&&\left|V(|u|^2)-\GG(|u|^2)\right|\leq C\int_0^{+\infty}\int_{\SS^2}\frac{\kappa(r,r')}{|\sigma-\sigma'|^{1+\eta}}|u(r'\sigma')|^2\,d\sigma'r'^2dr'\\
&&\quad \leq C\int_0^{+\infty}\left((\chi(r-1))^\eta+(\chi(r'-1))^\eta\right)\left(1+\frac{1}{r'}\right)\|u(r'\cdot)\|_{L^{2p}_\sigma}^{\frac{(1+\eta)p}{p-1}}\|u(r'\cdot )\|_{L^2_\sigma}^{\frac{(1-\eta)p-2}{p-1}}r'^2dr'\\
&&\quad \leq  C\int_{1/2}^{+\infty}\left((\chi(r-1))^\eta+(\chi(r'-1))^\eta\right)\|u(r'\cdot)\|_{L^{2p}_\sigma}^{\frac{(1+\eta)p}{p-1}}\|u(r'\cdot )\|_{L^2_\sigma}^{\frac{(1-\eta)p-2}{p-1}}r'^2dr'\\
&&\quad \quad + C\int_0^{1/2}\frac{1}{r'}\|u(r'\cdot)\|_{L^{2p}_\sigma}^{\frac{(1+\eta)p}{p-1}}\|u(r'\cdot )\|_{L^2_\sigma}^{\frac{(1-\eta)p-2}{p-1}}r'^2dr'
\eee
where we used the H\"older inequality and where $p>\frac{2}{1-\eta}$. The first integral in the last inequality can be bounded as in Step 2, by using the Sobolev embedding \fref{em1}:
\bee 
&&\int_{1/2}^{+\infty}\left((\chi(r-1))^\eta+(\chi(r'-1))^\eta\right)\|u(r'\cdot)\|_{L^{2p}_\sigma}^{\frac{(1+\eta)p}{p-1}}\|u(r'\cdot )\|_{L^2_\sigma}^{\frac{(1-\eta)p-2}{p-1}}r'^2dr'\\
&&\qquad \qquad \qquad \qquad
\qquad \leq (\chi(r-1))^\eta \|\lnar u\|_{L^2}^2+\|\chi(r'-1))^\eta \lnar u\|_{L^2}\|\lnar u\|_{L^2}.
\eee
The second integral can be bounded as in Step 3: by using \fref{em1}, \fref{em2} and \fref{confresu}, we get, for $\eta$ small enough and $p$ large enough,
$$\int_0^{1/2}\frac{1}{r'}\|u(r'\cdot)\|_{L^{2p}_\sigma}^{\frac{(1+\eta)p}{p-1}}\|u(r'\cdot )\|_{L^2_\sigma}^{\frac{(1-\eta)p-2}{p-1}}r'^2dr'
\leq C\eps ^\gamma\left(\|\lnar u\|_{L^2}^2+\eps^2\|u\|_{\calB^1}^2\right).$$
Finally, the first term in \fref{mik25} can be estimated as follows:
\bea
&&\left(\int_0^{+\infty}\left\|V(|u|^2)-\GG(|u|^2)\right\|_{L^\infty_\sigma}^2\left\|\lnar u\right\|_{L^2_\sigma}^2r^2dr\right)^{1/2}\nonumber\\
&&\qquad \leq C\|\chi(r'-1))^\eta \lnar u\|_{L^2}\|\lnar u\|_{L^2}^2+C\eps ^\gamma\left(\|\lnar u\|_{L^2}^3+\eps^3\|u\|_{\calB^1}^3\right).\qquad\quad
\label{mik26}
\eea
Let us now estimate the second term in \fref{mik25}. Setting
$$w(r',\sigma)=\int_{\SS^2}\frac{1}{|\sigma-\sigma'|^{1+\eta}}\lnar(|u|^2)(r'\sigma')d\sigma',$$
we deduce from \fref{ecritdelta}, \fref{estidelta} and from the Minkowski inequality  that
\bee
&&\int_0^{+\infty}\left\|V(\lnar(|u|^2))-\GG(\lnar(|u|^2))\right\|_{L^4_\sigma }^2\left\|u\right\|_{L^4_\sigma}^2r^2dr\\
&& \quad \leq C\left(\int_0^{+\infty}(1+\frac{1}{\sqrt{r'}})\chi(r'-1)^\eta\|w(r',\cdot)\|_{L^4_\sigma}r'^2dr'\right)^2\left(\int_0^{+\infty}(1+\frac{1}{\sqrt{r}})^2\left\|u(r\cdot)\right\|_{L^4_\sigma}^2r^2dr\right)\\
&&\quad +C\left(\int_0^{+\infty}(1+\frac{1}{\sqrt{r'}})\|w(r',\cdot)\|_{L^4_\sigma}r'^2dr'\right)^2\left(\int_0^{+\infty}(1+\frac{1}{\sqrt{r}})^2\chi(r-1)^{2\eta}\left\|u(r\cdot)\right\|_{L^4_\sigma}^2r^2dr\right).
\eee
By Hardy-Littlewood-Sobolev, one has
$$\|w(r',\cdot)\|_{L^4_\sigma}\leq C\|\lnar (|u|^2)\|_{L^\frac{4}{3-2\eta}_\sigma}\leq C\|\lnar u\|_{L^2}\|u\|_{L^\frac{4}{1-2\eta}_\sigma}$$
and, by adapting the proof of \fref{em3}, one can prove that there exists $\gamma>0$ such that, for $\eta$ small enough,
\be
\label{em3bis}
\left\|\frac{u}{\sqrt{r}}\right\|_{L^2_rL^\frac{4}{1-2\eta}_\sigma}\leq C\|\lnar u\|_{L^2}+C\eps^{1+\gamma}\|u\|_{\calB^1}.
\ee
Finally, using \fref{em1} and \fref{em3bis}, one gets
\bea
&&\left(\int_0^{+\infty}\left\|V(\lnar(|u|^2))-\GG(\lnar(|u|^2))\right\|_{L^4_\sigma }^2\left\|u\right\|_{L^4_\sigma}^2r^2dr\right)^{1/2}\nonumber\\
&& \quad \qquad \qquad \leq C\|\chi(r-1))^\eta \lnar u\|_{L^2}\left (\|\lnar u\|_{L^2}^2+C\eps^{2+2\gamma}\|u\|_{\calB^1}^2\right )\quad 
\label{mik27}
\eea
and the claim \fref{chi} can be deduced from \fref{mik25}, \fref{mik26}, \fref{mik27}, the Young inequality and the H\"older inequality.

We are in position to conclude the proof of \fref{poisdiff}. By using an interpolation inequality on the sphere and the properties of the truncation function $\chi$, we obtain
\bee
\|\chi(r-1)\lnar u\|_{L^2}&\leq& C\|\chi(r-1)u\|_{L^2}^{1/2}\|\lnar^2 u\|^{1/2}_{L^2}\\
&\leq& C\||r-1|^{\alpha/2}u\|_{L^2}^{1/\alpha}\|u\|_{L^2}^{1/2-1/\alpha}\|\lnar^2 u\|^{1/2}_{L^2}\\
&\leq&C\eps^{1/2}\left\|\Vc\left(\frac{r-1}{\eps}\right)^{1/2} u\right\|^{1/\alpha}\|u\|_{L^2}^{1/2-1/\alpha}\|\lnar^2 u\|^{1/2}_{L^2}\\
&\leq& C\eps^{1/2}\left(\eps \|u\|_{\calB^1}\right)^{1/\alpha}\|\lnar^2 u\|^{{1-1/\alpha}}_{L^2}
\eee
where we used Assumption \ref{H1} and the fact that $\alpha\geq 2$. Finally, inserting this estimate in \fref{chi} and using the Young inequality leads to \fref{poisdiff}, up to changing the value of $\gamma$.
\end{proof}

Let us now state another estimate, where we recall that the operator $H_r$ was defined in \fref{Hr}.
\begin{lemma}
\label{noyaupoisson2}
Let $u$ be such that $\lnar u\in L^2$ and $H_r^{s/2}\lnar u\in L^2$ with $s\geq 1$ an integer, then
\be
\left\|(1+\eps^2 H_r)^{s/2}\lnar\left(\GG(|u|^2)u\right)\right\|_{L^2}\leq C\|\lnar u\|_{L^2}^2\|(1+\eps^2 H_r)^{s/2} \lnar u\|_{L^2}.\label{pois3}
\ee
\end{lemma}
\begin{proof}
Since the operator $H_r$ only acts on the variable $r$ and since $\GG$ is independent of the variable $r$, we have, for all $s\geq 0$,
\bee
\|(1+\eps^2 H_r)^{s/2}\lnar\left(\GG(|u|^2)u\right)\|_{L^2}&=&\|\lnar\left(\GG(|u|^2)(1+\eps^2 H_r)^{s/2}u\right)\|_{L^2}\nonumber\\
&\leq & \|\lnar u\|_{L^2}^2 \|(1+\eps^2 H_r)^{s/2}\lnar u\|_{L^2},\label{j1}
\eee
where we used \fref{norme1} and the following inequality, that can be proved as \fref{pois2} with $s=1$:
$$
\left\|\lnar\left(\GG(|u|^2)v\right)\right\|_{L^2}\leq C\|\lnar u\|_{L^2}^2\|\lnar v\|_{L^2}.
$$
\end{proof}

We end this section with the following lemma, which will enable to deal with more convenient norms than $\left\|(1+\eps^2 H_r)^{s/2}u\right\|_{L^2}$.
\begin{lemma}
\label{lemequivnorm}
Assume that the confinement potential satisfies Assumption \pref{H1} and let $H_r$ be defined by \fref{Hr}. Then $H_r$ is a positive selfadjoint operator on $L^2(\RR_+,r^2dr)$ and, for every integer $s\geq 1$,
$$D((1+H_r)^{s/2})=\left\{u\in L^2(\RR_+,r^2dr):\, (\Vc^\eps)^{\frac{s-k}{2}}\pa_r^k(ru)\in L^2(\RR_+,dr), \,0\leq k\leq s\right\}.$$
Moreover, the following norms are equivalent, with constants independent of $\eps$:
$$\left\|(1+\eps^2 H_r)^{\frac{s}{2}}u\right\|_{L^2},$$
$$\|u\|_{L^2}+\eps^s\left\|\frac{1}{r}\pa_r^s(ru)\right\|_{L^2}+\eps^s\|(\Vc^\eps)^{\frac{s}{2}}u\|_{L^2},$$
$$\|u\|_{L^2}+\eps^s\sum_{k=0}^s\left\|(\Vc^\eps)^{\frac{s-k}{2}}\frac{1}{r}\pa_r^k(ru)\right\|_{L^2}\quad \mbox{and}\quad \|u\|_{L^2}+\eps^s\sum_{k=0}^s\left\|\frac{1}{r}\pa_r^k\left(r(\Vc^\eps)^{\frac{s-k}{2}}u\right)\right\|_{L^2},$$
where $\|\cdot\|_{L^2}$ denotes the $L^2(\RR_+,r^2dr)$ norm.
\end{lemma}
 Note that proving this result is equivalent to proving the equivalence of the following norms in dimension one:
$$\left\|(1+\eps^2 (-\pa_r^2+\Vc^\eps))^{\frac{s}{2}}u\right\|_{L^2(\RR_+,dr)},$$
$$\|u\|_{L^2(\RR_+,dr)}+\eps^s\left\|\pa_r^su\right\|_{L^2(\RR_+,dr)}+\eps^s\|(\Vc^\eps)^{\frac{s}{2}}u\|_{L^2(\RR_+,dr)},$$
$$\|u\|_{L^2(\RR_+,dr)}+\eps^s\sum_{k=0}^s\left\|(\Vc^\eps)^{\frac{s-k}{2}}\pa_r^k u\right\|_{L^2(\RR_+,dr)},$$
$$\|u\|_{L^2(\RR_+,dr)}+\eps^s\sum_{k=0}^s\left\|\pa_r^k\left(\Vc^\eps)^{\frac{s-k}{2}}u\right)\right\|_{L^2(\RR_+,dr)}.$$
Now, recalling that $\Vc^\eps(r)=\frac{1}{\eps^2}\Vc(\frac{r-1}{\eps})$, let us apply the dilation $r'=\frac{r-1}{\eps}.$ The above norms become
$$\left\|(1-\pa_r^2+\Vc)^{\frac{s}{2}}u\right\|_{L^2(]-\frac{1}{\eps},+\infty[)},$$
$$\|u\|_{L^2(]-\frac{1}{\eps},+\infty[)}+\left\|\pa_r^su\right\|_{L^2(]-\frac{1}{\eps},+\infty[)}+\|(\Vc)^{\frac{s}{2}}u\|_{L^2(]-\frac{1}{\eps},+\infty[)},$$
$$\|u\|_{L^2(]-\frac{1}{\eps},+\infty[)}+\sum_{k=0}^s\left\|(\Vc)^{\frac{s-k}{2}}\pa_r^k u\right\|_{L^2(]-\frac{1}{\eps},+\infty[)},$$
$$\|u\|_{L^2(]-\frac{1}{\eps},+\infty[)}+\sum_{k=0}^s\left\|\pa_r^k\left(\Vc)^{\frac{s-k}{2}}u\right)\right\|_{L^2(]-\frac{1}{\eps},+\infty[)}.$$
Hence, Lemma \ref{lemequivnorm} is a consequence of Proposition \ref{propapp} proved in  Appendix A.

\subsection{Approximation by an intermediate model}

In this subsection, we make a first step towards Theorem \ref{thm3}. We obtain a priori estimates for the singularly perturbed system \ref{schrod} and prove that it can be approximated by the following system, where we only pass to the limit in the nonlinear term:
\begin{equation}
  \label{schrodinter}
  i\pa_tw^\eps=-\Delta w^\eps+\Vc^\eps w^\eps+\GG\left(|w^\eps|^2\right)w^\eps,\qquad w^\eps(t=0)=u_0^\eps\,,
\end{equation}
where $\GG$ is defined by \fref{G2}.

\begin{proposition}
\label{prop1}
Assume that the confinement potential and the initial data satisfy Assumptions \pref{H1} and \pref{Hinit}. Then the following holds true.\\
(i) For all $\eps\in(0,1]$,  \fref{schrod} admits a unique solution $u^\eps$ in the energy space $\calC^0(\RR,\calB^1)$. Moreover, there exists $T>0$ such that
\be
\label{estinter2}
\sup_{\eps \in(0,1]}\left(\eps\|u^\eps\|_{L^\infty([-T,T],\calB^1)}+\|\lnar u^\eps\|_{L^\infty([-T,T],L^2)}\right)<+\infty.
\ee
(ii) There exists $T>0$ such that, for all $\eps\in(0,1]$, \fref{schrodinter} admits a unique solution $w^\eps\in \calC^0([-T,T],\calB^1)$ with $\lnar w^\eps\in \calC^0([-T,T],L^2)$, and with a uniform bound:
\be
\label{estinter4}
\sup_{\eps \in(0,1]}\left(\eps\|w^\eps\|_{L^\infty([-T,T],\calB^1)}+\|\lnar w^\eps\|_{L^\infty([-T,T],L^2)}\right)<+\infty.
\ee
(iii)
Assume that $\eps_0\in (0,1]$ and $T>0$ are such that
\be
\label{estinter4bis}
\sup_{\eps \in(0,\eps_0]}\left(\eps\|w^\eps\|_{L^\infty([-T,T],\calB^1)}+\|\lnar w^\eps\|_{L^\infty([-T,T],L^2)}\right)<+\infty.
\ee
Then, one has
\be
\label{estinter3}
\lim_{\eps\to 0}\|\lnar (u^\eps-w^\eps)(t)\|_{L^\infty([-T,T],L^2)}=0.
\ee
\end{proposition}
\begin{proof}
The well-posedness of the Cauchy problem in the energy space $\calB^1$ for \fref{schrod}, for all fixed $\eps>0$, is very standard. It can be proved by using standard techniques: local in time existence by a Banach fixed-point procedure, then global existence thanks to the conservation laws \fref{consL2} and \fref{consenergy}. 

\bs
\ni
{\em Step 1: a priori estimate \fref{estinter2}.} Let us now prove the existence of $T>0$ such that the a priori estimate \fref{estinter2} holds. Thanks to Assumption \ref{Hinit}, we have
\be
\label{prempart}
\|\na u_0^\eps\|_{L^2}^2+\|(\Vc^\eps)^{1/2}u_0^\eps\|_{L^2}^2\leq \|u_0^\eps\|_{\calB^1}^2\leq \frac{C}{\eps^2}
\ee
and
$$\|\na V(|u_0^\eps|^2)\|_{L^2}^2=\int_{\RR^3} V(|u_0^\eps|^2)|u_0^\eps|^2dx\leq \|V(|u_0^\eps|^2)\|_{L^\infty}\|u_0^\eps\|^2_{L^2}\leq C,
$$
where we used \fref{norme4}. Thus, the energy conservation law \fref{consenergy} yields, for all $t$,
\be
\label{eee1}
\eps^2\|u^\eps(t)\|_{\calB^1}^2\leq \eps^2 \|u_0\|_{\calB^1}^2+\frac{\eps^2}{2}\|\na V(|u_0^\eps|^2)\|_{L^2}^2\leq C.
\ee
Apply now the operator $\lnar=\sqrt{1-\Delta_\sigma}$ to the Schr\"odinger equation \fref{schrod}. Since this operator commutes with the Hamiltonian, we get for all $t$
$$\begin{array}{ll}
\ds \|\lnar u^\eps(t)\|_{L^2}&\ds \leq \|\lnar u_0^\eps\|_{L^2}+\int_{-|t|}^{|t|} \left\|\lnar\left(V(|u^\eps|^2)u^\eps\right)\right\|_{L^2}(\tau)\,d\tau\\[3mm]
&\ds \leq C+ C\int_{-|t|}^{|t|}\left(\|\lnar u^\eps\|_{L^2}^2+\eps^2\,\|u^\eps\|_{\calB^1}^2\right)\|\lnar u^\eps\|_{L^2}d\tau,
\end{array}
$$
where we used Assumption \ref{Hinit} and \fref{pois1}. Using the estimate \fref{eee1} and a standard bootstrap lemma, this yields a local in time estimate: there exist $T$ and $C_T$ independent of $\eps$ such that 
\begin{equation}
  \label{eee2}
  \forall t\in [-T,T]\qquad \|\lnar u^\eps(t)\|_{L^2}\leq C_T.
\end{equation}
The proof of Item {\em (i)} of Proposition \ref{prop1} is complete.

\bs
\ni
{\em Step 2: the Cauchy problem for the intermediate model}.
Let us now consider the Cauchy problem for \fref{schrodinter}. By using \fref{pois2} with $s=1$, it is easy to prove by a fixed-point procedure that, for all $\eps>0$, \fref{schrodinter} admits a unique maximal solution $w^\eps$ such that $\lnar w^\eps\in \calC^0((-T^\eps,T^\eps),L^2)$.
Note that, by a bootstrap argument similar as above, one can prove that $T^\eps$ is bounded from below independently of $\eps$: there exists $T>0$ such that
$$\sup_{\eps\in(0,\eps_0]}\|\lnar w^\eps\|_{L^\infty([-T,T],L^2)}<+\infty.$$
Moreover, this solution $w^\eps$ also belongs to the energy space for all time and satisfies the mass and energy conservation laws:
\be
\label{consL2inter}
\|w^\eps(t)\|_{L^2}^2=\|u^\eps_0\|_{L^2}^2
\ee
\bee
\|\na w^\eps(t)\|_{L^2}^2+\|(\Vc^\eps)^{1/2}w^\eps(t)\|_{L^2}^2+\frac{1}{2}\int_{\RR^3} \GG(|w^\eps(t,x)|^2)|w^\eps(t,x)|^2dx\qquad\qquad &&\nonumber\\=\|\na u_0^\eps\|_{L^2}^2+\|(\Vc^\eps)^{1/2}u_0^\eps\|_{L^2}^2+\frac{1}{2}\int_{\RR^3} \GG(|u_0^\eps|^2)|u_0^\eps|^2dx.&&\label{consenergyinter}
\eee
Notice that this energy is finite and of order $\frac{1}{\eps^2}$ at the initial time. Indeed, the two first terms in the right-hand side are bounded by \fref{prempart} and the third term is bounded thanks to Assumption \ref{Hinit} and \fref{norme1}:
$$
\int_{\RR^3} \GG(|u_0^\eps|^2)|u_0^\eps|^2dx\leq \|\GG(|u_0^\eps|^2)\|_{L^\infty}\|u_0^\eps\|^2_{L^2}\leq C.
$$
Notice also that, to solve this problem  \fref{schrodinter}, it was crucial to bound the $L^\infty$ norm of $\GG(|u|^2)$, at the initial time and during the evolution, locally in time. Thanks to \fref{norme1}, such estimate is available as soon as $\lnar u$ belongs to $L^2$. This is the reason why we introduced Assumption \fref{init1} on the initial data. It is not clear whether the Cauchy problem for \fref{schrodinter} is well-posed on $H^1$ (or $\calB^1$) only.

\bs
\ni
{\em Step 3: regularization and approximation result.} 
Let us now prove \fref{estinter3}.  Let $\eps_0>0$ and $T>0$ be such that a uniform bound \fref{estinter4bis} holds. We set
$$M:=1+\sup_{\eps\in(0,\eps_0]}\left(\eps\|w^\eps\|_{L^\infty([-T,T],\calB^1)}+\|\lnar w^\eps\|_{L^\infty([-T,T],L^2)}\right)<+\infty.$$
Remark in particular that $\|\lnar u_0^\eps\|_{L^2}\leq M<2M$. Let $T_1^\eps$ be defined as follows:
$$
T_1^\eps=\sup\left\{\tau\in (0,T]:\,\|\lnar u^\eps\|_{L^\infty([-\tau,\tau],L^2)}<2M\right\}.
$$
From Step 1, we know that $T_1^\eps$ is bounded from below: there exists $T_1>0$ such that $T_1^\eps>T_1$ for all $\eps\in(0,\eps_0]$. Moreover, by a continuity argument, one can prove that if $T_1^\eps<T$, then
$$\|\lnar u^\eps\|_{L^\infty([-T_1^\eps,T_1^\eps],L^2)}=2M.$$

Let $\delta>0$. From Assumption \fref{init2} on the initial data, one can define a regularized subsequence $u_0^{\eps,\delta}$ of the initial data such that
\be
\label{regu1}
\sup_\eps \left( \|\lnar^2 u_0^{\eps,\delta}\|_{L^2}+\eps\|u_0^{\eps,\delta}\|_{\calB^1}\right)<+\infty
\ee
and
\be
\label{regu2}
\limsup_{\eps\to 0}\|\lnar(u_0^\eps-u_0^{\eps,\delta})\|\leq \delta.
\ee
Let $u^{\eps,\delta}$, $w^{\eps,\delta}$ be the solutions of \fref{schrod} and \fref{schrodinter} with $u_0^{\eps,\delta}$ as initial data. By standard arguments, using \fref{poisuv2} and \fref{poisuv1}, one can prove that these solutions depend continuously on the initial data: for $\delta$ small enough, \fref{regu1}, \fref{regu2} yield
\bea
\label{convdelta2}\limsup_{\eps\to 0}\|\lnar(w^{\eps,\delta}-w^\eps)\|_{L^\infty([-T,T],L^2)}&\leq &C\delta ,\\
\label{convdelta}\limsup_{\eps\to 0}\|\lnar(u^{\eps,\delta}-u^\eps)\|_{L^\infty([-T_1^\eps,T_1^\eps],L^2)}&\leq &C\delta,
\eea
where $C$ is independent of $\eps$ and where we used that $\eps \|u^{\eps,\delta}\|_{\calB^1}$ and $\eps \|u^{\eps}\|_{\calB^1}$ are uniformly bounded.
In particular, if $\eps_0$ and $\delta$ have been initially chosen small enough, we have the estimates
\be
\label{xx}
\sup_{\eps\in(0,\eps_0]}\|\lnar u^{\eps,\delta}\|_{L^\infty([-T_1^\eps,T_1^\eps],L^2)}+\sup_{\eps\in(0,\eps_0]}\|\lnar w^{\eps,\delta}\|_{L^\infty([-T,T],L^2)}\leq 4M.
\ee
Next, applying the operator $\lnar ^2$ to \fref{schrod} and \fref{schrodinter}, then using the tame estimates \fref{pois1} and \fref{pois2}, with $s=2$, one deduces thanks to \fref{xx} and the Gronwall lemma that, for all $\eps\in(0,\eps_0]$,
\be
\label{normeH2}
\|\lnar ^2 u^{\eps,\delta}\|_{L^\infty([-T_1^\eps,T_1^\eps],L^2)}+ \|\lnar ^2w^{\eps,\delta}\|_{L^\infty([-T,T],L^2)}\leq C_M\|\lnar^2u_0^{\eps,\delta}\|_{L^2}\leq C_{M,\delta}
\ee
where $C_M$ and $C_{M,\delta}$ are generic constants which only depend, respectively, on $M$ and on $(M,\delta)$, and where we used the bound \fref{regu1}.

Let us now write the equation satisfied by the difference $z=u^{\eps,\delta}-w^{\eps,\delta}$:
\bea
i\pa_tz=-\Delta z+\Vc^\eps z&+&V(|u^{\eps,\delta}|^2)u^{\eps,\delta}-\GG(|u^{\eps,\delta}|^2)u^{\eps,\delta}\nonumber\\
&+&\GG(|u^{\eps,\delta}|^2)u^{\eps,\delta}-\GG(|w^{\eps,\delta}|^2)w^{\eps,\delta}.\label{eqz}
\eea
From \fref{poisdiff} and \fref{normeH2} (and using also the energy estimates), we deduce that
$$\|\lnar(V(|u^{\eps,\delta}|^2)u^{\eps,\delta}-\GG(|u^{\eps,\delta}|^2)u^{\eps,\delta})\|_{L^2}\leq \eps^\gamma\,C_{M,\delta}.$$
Moreover, since $\lnar u^{\eps,\delta}$ and $\lnar w^{\eps,\delta}$ are uniformly bounded in $L^2$, one deduces from \fref{poisuv2} that, for all $t$ and $\eps$,
$$\|\lnar(\GG(|u^{\eps,\delta}|^2)u^{\eps,\delta}-\GG(|w^{\eps,\delta}|^2)w^{\eps,\delta})\|_{L^2}\leq C_M\|\lnar z(t)\|_{L^2}.$$
Finally, since $z(t=0)=0$, applying $\lnar$ to \fref{eqz} leads to
$$\|\lnar z(t)\|_{L^2}\leq \eps^\gamma \,C_{M,\delta}\,T+C_M\int_{-|t|}^{|t|} \|\lnar z(\tau)\|_{L^2}d\tau,$$
and a Gronwall lemma enables to conclude that
$$\|\lnar z\|_{L^\infty([-T_1^\eps,T_1^\eps],L^2)}\leq \eps^\gamma\,C_{M,\delta} \,T\,e^{C_M T}.$$
From this inequality and from \fref{convdelta2}, \fref{convdelta}, letting $\delta$ tend to zero, one deduces that
$$\lim_{\eps\to 0} \|\lnar (u^\eps-w^\eps)\|_{L^\infty([-T_1^\eps,T_1^\eps],L^2)}=0.$$
Finally, by fixing $\eps_1\in (0,\eps_0]$ small enough such that
$$\sup_{\eps\in(0,\eps_1]}\|\lnar (u^\eps-w^\eps)\|_{L^\infty([-T_1^\eps,T_1^\eps],L^2)}<\frac{M}{2},$$
we ensure that
$$\sup_{\eps\in(0,\eps_1]}\|\lnar w^\eps\|_{L^\infty([-T_1^\eps,T_1^\eps],L^2)}\leq \frac{3M}{2}<2M$$
(since $M>0$). We deduce from this inequality that, for all $\eps\in(0,\eps_1]$, $T_1^\eps= T$. The proof of the proposition is complete.
\end{proof}

\subsection{Proof of Theorem \ref{thm3}}

In order to prove Theorem \ref{thm3}, it remains to prove that the solution of the intermediate model \fref{schrodinter} can be approximated by the solution of the limit model \fref{schrodlim}. Notice that \fref{schrodinter} reads
$$i\pa_tw^\eps=H_r w^\eps-\frac{1}{r^2}\Delta_\sigma w^\eps+\GG\left(|w^\eps|^2\right)w^\eps,\qquad w^\eps(t=0)=u_0^\eps\,.$$
\begin{proposition}
\label{prop2}
Assume that the confinement potential and the initial data satisfy Assumptions \pref{H1} and \pref{Hinit}. Then the following holds true.\\
(i) For all $\eps>0$, the limit system \fref{schrodlim} admits a unique global solution $v^\eps$ such that $\lnar v^\eps\in \calC^0(\RR,L^2)$. Moreover, the following conservation laws are satisfied for all $t\in\RR$:
\bea
\label{conslimit1}
\|v^\eps(t)\|^2_{L^2}&=&\|u^\eps_0\|^2_{L^2},\\
\label{conslimit2}
\|\pa_rv^\eps(t)\|_{L^2}^2+\|(\Vc^\eps)^{1/2} v^\eps(t)\|_{L^2}^2&=&\|\pa_ru_0^\eps|_{L^2}^2+\|(\Vc^\eps)^{1/2} u_0^\eps\|_{L^2}^2,\\
\|\na_\sigma v^\eps\|_{L^2}^2+\frac{1}{2}\int_{\RR^3}\GG\left(\left|v^\eps\right|^2\right)\left|v^\eps\right|^2dx&=&\|\na_\sigma u_0^\eps\|_{L^2}^2+\frac{1}{2}\int_{\RR^3}\GG\left(\left|u_0^\eps\right|^2\right)\left|u_0^\eps\right|^2dx.\quad\qquad\label{conslimit3}
\eea
(ii) Let $T>0$. Then there exists $\eps_0>0$ such that the intermediate model \fref{schrodinter} admits a unique solution $w^\eps$ on $[-T,T]$ with a uniform bound \fref{estinter4bis}. Moreover, one has
\be
\label{convinter}
\lim_{\eps\to 0}\|\lnar (w^\eps-v^\eps)(t)\|_{L^\infty([-T,T],L^2)}=0.
\ee
\end{proposition}
\begin{proof}
Thanks to \fref{pois2}, \fref{poisuv2} and \fref{pois3} (with $s=1$) it is easy to prove that the Cauchy problem \fref{schrodlim} is well-posed locally in time. In fact, the solution will be global thanks to \fref{conslimit1}, \fref{conslimit2} and \fref{conslimit3}. Let us now prove these conservation laws. The first one \fref{conslimit1} is the standard conservation of the $L^2$ norm. The second one \fref{conslimit2} is the conservation of the $L^2$ norm for the equation satisfied by $H_r^{1/2}v^\eps$:
$$ i\pa_tH_r^{1/2}v^\eps=H_rH_r^{1/2}v^\eps-\Delta_\sigma H_r^{1/2}v^\eps+\GG\left(|v^\eps|^2\right)H_r^{1/2}v^\eps,\quad H_r^{1/2}v^\eps(t=0)=H_r^{1/2}u_0^\eps\,,
$$
recalling that $\GG(\cdot)$ is independent of $r$ and that
$$H_r=-\frac{1}{r^2}\pa_r(r^2\pa_r)+\Vc^\eps\,,\qquad \|(H_r)^{1/2}u\|_{L^2}^2=\|\pa_ru\|_{L^2}^2+\|(\Vc^\eps)^{1/2} v^\eps\|_{L^2}^2.$$
The third identity \fref{conslimit3} is obtained by multiplying \fref{schrodlim} by $\pa_t\overline{v^\eps}$, integrating on $\RR^3$, taking the real part of the equation and finally using \fref{conslimit2}.

Let us now prove Item {\em (ii)}. Let $T>0$ and denote
$$M_0=1+\sup_{\eps\in (0,1]}\left(\|u_0^\eps\|_{L^2}^2+\|\na_\sigma u_0^\eps\|_{L^2}^2+\frac{1}{2}\int_{\RR^3}\GG\left(\left|u_0^\eps\right|^2\right)\left|u_0^\eps\right|^2dx\right)<+\infty.$$
By \fref{conslimit1} and \fref{conslimit3}, we have, for all $\eps>0$,
\be
\label{ppqM0}
\|\lnar v^\eps\|_{L^\infty([-T,T],L^2)}\leq M_0.
\ee
By Proposition \ref{prop1}, the Cauchy problem for the intermediate model \fref{schrodinter} is locally well-posed for $0<\eps<1$, with a uniform bound of the form \fref{estinter4}. Denote by $w^\eps$ its solution and set
\be
\label{defTeps}
T_0^\eps=\sup\left\{\tau\in (0,T]:\,\|\lnar w^\eps\|_{L^\infty([-\tau,\tau],L^2)}<2M_0\right\}.
\ee
We already know that $T_0^\eps$ is bounded from below: there exists $T_0>0$ such that $T_0^\eps>T_0$ for all $\eps\in(0,1]$. Moreover, by a continuity argument, one can prove that if $T_0^\eps<T$, then
$$\|\lnar w^\eps\|_{L^\infty([-T_0^\eps,T_0^\eps],L^2)}=2M_0.$$

Let us now regularize the initial data as follows. From \fref{init2} again, we deduce that, for all $\delta>0$, there exists a subsequence of the initial data $u_0^\eps$, still denoted $u^{\eps,\delta}_0$, such that
\be
\label{regu3}
\sup_\eps \left( \|\lnar^6 u_0^{\eps,\delta}\|_{L^2}+\|(1+\eps H_r)^{3/2}\lnar u_0^{\eps,\delta}\|_{L^2}\right)<+\infty
\ee
and
\be
\label{regu4}
\limsup_{\eps\to 0}\|\lnar(u_0^\eps-u_0^{\eps,\delta})\|\leq \delta.
\ee
We denote by $v^{\eps,\delta}$ the corresponding solution of \fref{schrodlim}. By standard arguments, using \fref{regu3}, \fref{regu4} and \fref{poisuv2}, one can prove that
\be
\label{convdeltabis2}\limsup_{\eps\to 0}\|\lnar(v^{\eps,\delta}-v^\eps)\|_{L^\infty([-T,T],L^2)}\leq C\delta.
\ee
In particular, for $\delta$ and $\eps_0$ small enough we have the estimate
\be
\label{xxbis}
\sup_{\eps\in(0,\eps_0]}\|\lnar v^{\eps,\delta}\|_{L^\infty([-T,T],L^2)}\leq 4M_0.
\ee
Next, applying the operator $\lnar ^6$ to \fref{schrodlim}, then using the tame estimate \fref{pois2} with $s=6$,
\bee
\|\lnar ^6v^{\eps,\delta}(t)\|_{L^2}&\leq& C\|\lnar^6u_0^{\eps,\delta}\|_{L^2}+C\int_{-|t|}^{|t|} \|\lnar v^{\eps,\delta}(\tau)\|_{L^2}^2\|\lnar ^6v^{\eps,\delta}(\tau)\|d\tau\\
&\leq&C_{M_0,\delta}+C_{M_0}\int_{-|t|}^{|t|} \|\lnar ^6v^{\eps,\delta}(\tau)\|d\tau.
\eee
Hence the Gronwall lemma yields
\be
\|\lnar ^6v^{\eps,\delta}\|_{L^\infty([-T,T],L^2)}\leq C_{M_0,\delta}.\label{normeH2bis}
\ee
Similarly, applying the operator $(1+\eps H_r)^{3/2}\lnar$ to \fref{schrodlim}, then using the estimate \fref{pois3}, with $s=4$ leads to
\bee
\|(1+\eps^2H_r)^{3/2}\lnar v^{\eps,\delta}(t)\|_{L^2}&\leq& C\|(1+\eps^2H_r)^{3/2}\lnar u_0^{\eps,\delta}\|_{L^2}\\
&&\hspace*{-1cm}+C\int_{-|t|}^{|t|} \|\lnar v^{\eps,\delta}(\tau)\|_{L^2}^2\|(1+\eps^2H_r)^{3/2}\lnar v^{\eps,\delta}(\tau)\|d\tau\\
&\leq&C_{M_0,\delta}+C_{M_0,\delta}\int_{-|t|}^{|t|} \|(1+\eps^2H_r)^{3/2}\lnar v^{\eps,\delta}(\tau)\|d\tau.
\eee
The crucial point for this estimate was that the operators $(1+\eps H_r)^{3/2}\lnar$ and $H_r-\Delta_\sigma$ commute together (whereas $(1+\eps H_r)^{3/2}\lnar$ does not commute with the complete Laplace operator $\Delta$ that appears in the intermediate model \fref{schrodinter}). Hence the Gronwall lemma yields
\be
\|(1+\eps^2H_r)^{3/2}\lnar v^{\eps,\delta}\|_{L^\infty([-T,T],L^2)}\leq  C_{M_0,\delta}.\label{normeH2ter}
\ee

With these estimates, we are now ready to conclude. Let us introduce a smooth function $\chi$, defined on $\RR_+$, such that $0\leq \chi\leq 1$, $\chi(r)=0$ for $r\leq 1/3$ and $\chi(r)=1$ for $r\geq 2/3$. Since the support of $1-\chi$ is $\{r\leq 2/3\}$, one has
\bea
\|(1-\chi)\lnar v^{\eps,\delta}\|_{L^2}&\leq& C \||r-1|^\alpha\lnar v^{\eps,\delta}\|_{L^2}\nonumber\\
&\leq&C\eps^\alpha \|\eps^2\Vc^\eps \lnar v^{\eps,\delta}\|_{L^2}\nonumber\\
&\leq&C\eps^\alpha \|(1+\eps^2H_r) \lnar v^{\eps,\delta}\|_{L^2}\leq \eps^\alpha C_{M_0,\delta}
\label{estimdiffchi}\eea
where we used Assumption \fref{superharm}, Lemma \ref{lemequivnorm} and \fref{normeH2ter}. Moreover, the function $\chi v^{\eps,\delta}$ satisfies the equation
$$i\pa_t(\chi v^{\eps,\delta})=H_r(\chi v^{\eps,\delta})-\Delta_\sigma (\chi v^{\eps,\delta})+\GG\left(|v^{\eps,\delta}|^2\right)(\chi v^{\eps,\delta})+R^\eps$$
where the remainder
$$R^\eps=-2\chi'\pa_rv^{\eps,\delta}-\left(\chi''+\frac{2}{r}\chi'\right)v^{\eps,\delta}$$
can be estimated as follows:
\bea
\|\lnar R^\eps\|_{L^2}&\leq& C \left\||r-1|^\alpha\lnar \frac{1}{r}\pa_r (rv^{\eps,\delta})\right\|_{L^2}+C \||r-1|^\alpha\lnar v^{\eps,\delta}\|_{L^2}\nonumber\\
&\leq &C \eps^{\alpha-1}\left\|\eps^3\Vc^\eps\lnar \frac{1}{r}\pa_r (rv^{\eps,\delta})\right\|_{L^2}+C \eps^\alpha\|\eps^2\Vc^\eps\lnar v^{\eps,\delta}\|_{L^2}\qquad \nonumber\\
&\leq& C\eps^{\alpha-1} \|(1+\eps^2H_r)^{3/2} \lnar v^{\eps,\delta}\|_{L^2}\\
&\leq& \eps^{\alpha-1}C_{M_0,\delta}
\label{estimremainder}
\eea
where we used again Assumption \fref{superharm}, Lemma \ref{lemequivnorm} and \fref{normeH2ter} and the fact that the supports of the functions $\chi'$ and $\chi''$ are $\{1/3\leq r\leq 2/3\}$.

Let us now estimate the difference $y=\chi v^{\eps,\delta}-w^\eps$. This function satisfies the equation
\bea
i\pa_t\lnar y&=&H_r\lnar y-\frac{1}{r^2}\Delta_\sigma \lnar y-\chi(r)\frac{r^2-1}{r^2}\Delta_\sigma \lnar v^{\eps,\delta}\nonumber\\
&& +\lnar\left(\GG(|v^{\eps,\delta}|^2)\chi v^{\eps,\delta}-\GG(|w^{\eps,\delta}|^2)w^\eps\right)+\lnar R^\eps\label{eqy}
\eea
with $y(0)=u_0^{\eps,\delta}-u_0^\eps$. By \fref{poisuv2}, \fref{defTeps}, \fref{xxbis} and \fref{estimdiffchi}, for $t\leq T_0^\eps$ we have
$$\left\|\lnar\left(\GG(|v^{\eps,\delta}|^2)\chi v^{\eps,\delta}-\GG(|w^\eps|^2)w^{\eps,\delta}\right)(t)\right\|_{L^2}\leq C_{M_0}\|\lnar y(t)\|_{L^2}+\eps^\alpha C_{M_0,\delta}.$$
Moreover, by interpolating and using that $\chi(r)$ vanishes near 0, one gets
\bee
\left\|\chi(r)\frac{r^2-1}{r^2}\Delta_\sigma \lnar v^{\eps,\delta}\right\|_{L^2}&\leq&C\left\|(r-1) v^{\eps,\delta}\right\|_{L^2}^{1/2}\left\|\lnar^6 v^{\eps,\delta}\right\|_{L^2}^{1/2}\\
&\leq &\eps^{1/2} C_{M_0,\delta}\left\|\Vc\left(\frac{r-1}{\eps}\right)^{1/\alpha} v^{\eps,\delta}\right\|_{L^2}^{1/2}\\
&\leq&\eps^{1/2}C_{M_0,\delta} \left\|\Vc\left(\frac{r-1}{\eps}\right)^{1/2} v^{\eps,\delta}\right\|_{L^2}^{1/\alpha}\|v^{\eps,\delta}\|_{L^2}^{1/2-1/\alpha}\\
&\leq&\eps^{1/2} C_{M_0,\delta}\left(\|(1+\eps^2 H_r)^{1/2}v^{\eps,\delta}\|_{L^2}\right)^{1/\alpha}\|v^{\eps,\delta}\|_{L^2}^{1/2-1/\alpha}\\
&\leq& \eps^{1/2}C_{M_0,\delta}.
\eee
In this series of inequalities, we used \fref{normeH2bis}, Assumption \fref{superharm}, a H\"older inequality (note that $\alpha>2$) and, finally, the conservation laws \fref{conslimit1} and \fref{conslimit2} for the regularized function $v^{\eps,\delta}$: 
$$\|(1+\eps^2 H_r)^{1/2}v^{\eps,\delta}\|_{L^2}^2=\|(1+\eps^2 H_r)^{1/2}u_0^{\eps,\delta}\|_{L^2}^2\leq C.$$

Finally, the $L^2$ estimate for \fref{eqy} yields
\bee
\|\lnar y(t)\|_{L^2}&\leq&\|\lnar(u_0^{\eps,\delta}-u_0^\eps)\|_{L^2} +\eps^{1/2}C_{M_0,\delta}+\eps^{\alpha-1}C_{M_0,\delta}\\
&&+C_{M_0}\int_{-|t|}^{|t|}\|\lnar y(\tau)\|_{L^2}d\tau.
\eee
We conclude by using the Gronwall lemma. We obtain
$$\limsup_{\eps\to 0}\|\lnar y\|_{L^\infty([-T_0^\eps,T_0^\eps],L^2)}\leq C\delta.$$
Hence, using \fref{convdeltabis2}, \fref{estimdiffchi} and letting $\delta$ tend to zero yields
$$\lim_{\eps\to 0}\|\lnar (v^\eps-w^\eps)\|_{L^\infty([-T_0^\eps,T_0^\eps],L^2)}=0.$$
In particular, from \fref{ppqM0}, we deduce that, for $\eps$ small enough, we have
$$\|\lnar w^\eps\|_{L^\infty([-T_0^\eps,T_0^\eps],L^2)}\leq \frac{2M_0}{3},$$
which implies that $T_0^\eps=T$. The proof of Proposition \ref{prop2} is complete.
\end{proof}

\bs
\ni
{\em Proof of Theorem \ref{thm3}}.

\ms
\ni
Theorem \ref{thm3} is a direct consequence of  Proposition \ref{prop2}, combined with Proposition \ref{prop1}. Indeed, Prop. \ref{prop2}, {\em (i)}, states that the limit system \fref{schrodlim} is globally well-posed. Let $T>0$. Prop. \ref{prop2}, {\em (ii)}, says that the intermediate system \fref{schrodinter} is well-posed on $[-T,T]$ (for $\eps$ small enough), is uniformly bounded and converges to \fref{schrodlim} as $\eps\to 0$. Therefore, Prop. \ref{prop1} {\em (iii)} can be applied thanks to this uniform bound: \fref{schrod} is asymptotically close to \fref{schrodinter}, thus also converges to \fref{schrodlim}.
\qed

\begin{appendix}

\section{Proof of Lemma \ref{lemequivnorm}}
\label{appA}

In this appendix we identify the norm on the domains of  iterates of the operator $H_r$ used in Section \ref{sect3}.
The lemma is a consequence of the following result.
\begin{proposition}
\label{propapp} Let $W\in C^\infty (\RR _+)$ be a real valued potential such that $W(x)\ge 1$ for every $x\in \RR $
and satisfying
\begin{equation}\label{estderiv}
\forall k\in \NN , \exists C_k>0 : \forall x\in \RR _+, \vert W^{(k)}(x)\vert \leq C_kW(x)\ .
\end{equation}
Consider the following unbounded operator on $L^2(\RR _+)$,
$$D(A)=\{ u\in H^2(\RR _+): u(0)=0 , Wu\in L^2(\RR  _+)\} \ \ ;\ \ Au:=-u''+Wu\ .$$
Then $A$ is a positive selfadjoint operator and, for every integer $s\ge 1$, 
\bee
D(A^{s/2})&=&\left\{ u\in L^2(\RR _+): \,\, W^{\frac{s-k}2}u^{(k)}\in L^2(\RR _+),\,\,0\le k\le s\right.\\
&&\qquad\left.    \mbox{and }\left (-\frac{d^2}{dx^2}+W\right )^pu(0)=0, \,0\le p\le \left[\frac {s-1}2\right]\ \right\}
\eee
and, on this space, the norm $\Vert A^{s/2}u\Vert _{L^2}$ is equivalent to
$$\sum _{k=0}^s \Vert W^{\frac{s-k}2}u^{(k)} \Vert _{L^2}\  ,$$
with constants only depending on the constants $C_k$ in (\ref{estderiv}), for $k$ in a finite set.
\end{proposition}
\begin{proof} For simplicity, we denote by $\Vert f\Vert $ the norm of $f$ in $L^2(\RR _+)$, and by $(f\vert g)$ the corresponding inner product.
We shall proceed in several steps.
\vskip0.25cm \noindent
{\sl Step 1.} The case $s=2$ and selfadjointness. In view of the definition of $A$, the symmetry identity 
$$(Au_1\vert u_2)=(u_1\vert Au_2)\ ,\ u_1,u_2\in D(A)$$
is merely an integration by parts. We now pass to a priori estimates. Firstly, integration by parts also implies
\begin{equation}\label{Auu}
(Au\vert u)=\Vert u'\Vert ^2 +\Vert \sqrt W u\Vert ^2 \ ,\ u\in D(A)\ .
\end{equation}
In view of the assumption on $W$, this implies in particular 
\begin{equation}\label{A1}
\Vert Au\Vert \ge \Vert u\Vert \ , \ u\in D(A)\ .
\end{equation}
Next we derive a more precise estimate on $\Vert Au\Vert $ by computing
$$\Vert Au\Vert ^2=\Vert u''\Vert ^2+\Vert Wu\Vert ^2-2{\rm Re}(u''\vert Wu)\ .$$
Introducing $\chi \in C^\infty _0(\RR )\ ,\ \chi \ge 0\ ,\ \chi =1 $ near $0$, we get
$$-(u''\vert Wu)=-\lim _{R\rightarrow\infty}\int _{\RR _+}\chi\left (\frac xR\right )u''(x)W(x)\overline {u(x)}\, dx$$
and,  after an integration by parts,
$$-(u''\vert Wu)=\lim _{R\rightarrow\infty}\int _{\RR _+} \left (\frac 1R\chi \, '\left (\frac xR\right )W(x)+\chi \left (\frac xR\right )W'(x)\right )u'(x)\overline {u(x)}+\chi \left (\frac xR\right )W(x)\vert u'(x)\vert ^2\, dx$$
Since $u'\in L^2$ and $Wu\in L^2$, the first term in the right hand side tends to $0$ as $R$ tends to infinity. Since moreover $W'=O(W)$, the second term has a limit. Consequently, the third term also has
a limit. By Fatou's lemma, we conclude that $\sqrt Wu'\in L^2$, and finally
\begin{equation}\label{normAu}
\Vert Au\Vert ^2=\Vert u''\Vert ^2+\Vert Wu\Vert ^2+2\Vert \sqrt Wu'\Vert ^2+2{\rm Re}(u'\vert W'u)\ .
\end{equation}
Because of the Cauchy--Schwarz inequality and of (\ref{estderiv}), 
\begin{eqnarray*}
2{\rm Re}(u'\vert W'u)\ge -2C_1\Vert u'\Vert \, \Vert Wu\Vert &\ge &-2C_1^2\vert u'\Vert ^2-\frac 12\Vert Wu\Vert ^2\\
&\ge & -\frac 12(\Vert u''\Vert ^2+\Vert Wu\Vert ^2)-C'_1\Vert u\Vert ^2 \ ,
\end{eqnarray*}
where $C_1'$ only depends on $C_1$. Combining this inequality with (\ref{A1}), we infer
\begin{equation}\label{normAu2}
\Vert Au\Vert \ge c_1(\Vert u''\Vert +\Vert Wu\Vert +\Vert \sqrt Wu'\Vert )\ ,
\end{equation}
where $c_1>0$ only depends on $C_1$. We therefore have proved the statement for $s=2$.
Let us use this inequality for proving that $A$ is selfadjoint. Recall that 
$$D(A^*)=\{ \psi \in L^2(\RR _+): \exists C>0, \forall u\in D(A), \vert (Au\vert \psi )\vert \le C\Vert u\Vert \} \ .$$
The symmetry of $A$ already implies that $D(A)\subset D(A^*)$ and that $A^*u=Au$ for every $u\in D(A)$.
Therefore we just have to prove that $D(A^*)\subset D(A)$. We claim that it is enough to prove that 
$\ker A^*=\{ 0\} $. Indeed, from estimate (\ref{normAu2}), it is easy to prove that the range of $A$ is a closed subspace
of $L^2$. Since its orthogonal is $\ker A^*$,  the cancellation of $\ker A^*$ would imply that $A$ is onto.
Consequently, for every $\psi \in D(A^*)$, there would exist $u\in D(A)$ such that $A^*\psi =Au$, namely $\psi -u \in \ker A^*$,
hence $\psi =u\in D(A)$.
\vskip 0.25cm \noindent

We now prove that  $\ker A^*=\{ 0\} $.  First of all, we observe that, for every $\psi \in D(A^*)$, in the distributional sense $\psi '' -W\psi \in L^2(\RR _+)$, hence
$\psi \in H^2((0,R))$ for every $R>0$. Moreover, by integration by parts,  for every $u\in D(A)$ supported into $[0,R]$ for some $R>0$,
we have
$$(Au\vert \psi )=u'(0)\overline {\psi (0)}+(u\vert -\psi ''+W\psi )= u'(0)\overline {\psi (0)}+O(\Vert u\Vert )$$
and testing the information $(Au\vert \psi )=O(\Vert u\Vert )$  on $u(x)=x\chi (nx)$ for large $n$ imposes $\psi (0)=0$. Assume moreover that $\psi \in \ker A^*$, namely that
$$\psi ''-W\psi =0\ .$$
By the Sobolev embedding, we infer that $\psi \in C^\infty (\RR _+)$.  
Set
$$v=\vert\psi \vert^2\ .$$
Then $v\in C^\infty (\RR _+)\cap L^1(\RR _+)$ and $v(0)=0$. Plugging the differential equation satisfied by $\psi $, we get
$$v''=2Wv+2\vert \psi '\vert ^2\ge 0\ .$$
In other word, $v$ is a convex function. Since $v$ is integrable at infinity, this implies that $v$ is non increasing and tends to $0$ at infinity.
Since $v(0)=0$, we conclude $v=0$ and hence $\psi =0$.
\vskip 0.25cm \noindent
{\sl Step 2.} The case $s=1$. The domain of $\sqrt A$ is characterized as the subspace of vectors $u\in L^2(\RR _+)$ such that there exists a sequence $(u_n)$ 
of $D(A)$ which tends to $u$ in $L^2$ and which is a Cauchy sequence for the norm 
$$N_1(v):=\sqrt{(Av\vert v)}\simeq \Vert v'\, \Vert +\Vert \sqrt Wv\Vert \ .$$
This clearly implies that, if $u\in D(\sqrt A)$, then  $u\in H^1_0(\RR _+)$ and $\sqrt Wu \in L^2(\RR _+)$. 
Conversely, if $u\in H^1_0(\RR _+)$ and $\sqrt Wu \in L^2(\RR _+)$, a simple cutoff shows that $u$ can be approximated in the $N_1$ norm by
elements of $H^1_0(\RR _+)$ with bounded supports. Then the claim reduces to the standard characterization of $H^1_0(\RR _+)$ as the closure of $C^\infty _0((0,\infty ))$ 
for the $H^1$ norm. 
\vskip 0.25cm \noindent
{\sl Step 3}. The general case. We just prove the description of $D(A^{\frac s2})$, the corresponding equivalence of norms being proved in the same way, by keeping track the constants.
We proceed by induction on $s$. Let $s\ge 3$ such that the claim is proved for every $s'\le s-1$. Then $u\in D(A^{\frac s2})$ if and only if $u\in D(A)$ and $Au\in D(A^{\frac {s-2}2})$. 
Using the induction hypothesis, the latter condition is equivalent to the following two conditions :
\begin{itemize}
\item $W^{\frac{s-2-k}2}(Au)^{(k)} \in L^2$ for every $k\le s-2$. Expanding $(Au)^{(k)}$ and using (\ref{estderiv}), we observe that this is equivalent to 
$$W^{\frac{s-k}2}u^{(k)}- W^{\frac{s-2-k}2}u^{(k+2)}\in L^2\ ,\ k\le s-2\ ,$$
since the error terms are controlled by the fact that $u\in D(A^{\frac{s-1}2})$. In the special case $k=0$, using again (\ref{estderiv}) and $u\in D(A^{\frac{s-1}2})$,
we observe that this condition is equivalent to 
$$-v''+Wv\in L^2$$
for $v:=W^{\frac{s-2}2}u$. Since $u\in D(A)$, $v\in H^2(0,R)$ for every $R>0$ and $v(0)=0$. Moreover, since $u\in D(A^{\frac{s-1}2})$ and by the induction hypothesis,
$v\in H^1(\RR _+)$. Hence, by computing $(Au\vert v)$ for $u\in D(A)$,  we have $v\in D(A^*)$, which, by the first step, 
implies $Wv\in L^2$, or $W^{\frac s2}u \in L^2$. Combining with the other conditions for $k\le s-2$, we eventually obtain $W^{\frac{s-k}2}u^{(k)} \in L2$ for $k\le s$.
\item The boundary conditions
$$ \left (-\frac{d^2}{dx^2}+W\right )^p(Au)(0)=0\ ,\  0\le p\le \left[\frac {s-3}2\right] \ .$$
Since $Au=-u''+Wu$ and since $u(0)=0$ from $u\in D(A)$, this leads to the claimed boundary conditions at rank $s$.
\end{itemize}
The proof is complete.

\end{proof}
\end{appendix}

 \ms
 \bs
\ni
{\bf Acknowledgements.}
The authors were supported by the Agence Nationale de la Recherche, ANR project  EDP dispersives. 
F. M\'ehats was also supported by the ANR project QUATRAIN and by the INRIA project IPSO.

\end{document}